\documentclass[letterpaper,11pt,reqno]{amsart}
\usepackage[portrait,margin=3cm]{geometry}

\usepackage{hyperref}
\usepackage{amsmath,amssymb,amsthm,amsfonts,amsbsy,latexsym,dsfont}
\usepackage{graphicx,color}

\usepackage{enumerate}

\newtheorem{definition}{Definition}[section]
\newtheorem{theorem}{Theorem}[section]
\newtheorem{proposition}[theorem]{Proposition}
\newtheorem{lemma}[theorem]{Lemma}
\newtheorem{remark}{Remark}[section]

\renewcommand{\le}{\leqslant}
\renewcommand{\ge}{\geqslant}

\newcommand{\ind}{\mathds{1}}
\newcommand{\eps}{\varepsilon}

\def\qed{ \hfill $\blacksquare$}
    \let\eps=\varepsilon
     \let\gl=\lambda

\newcommand{\cL}{\mathcal{L}}

\newcommand{\va}{\mathbf{a}}\newcommand{\vb}{\mathbf{b}}

\newcommand{\vp}{\mathbf{p}}

\newcommand{\vx}{\mathbf{x}}
\newcommand{\vy}{\mathbf{y}}

\newcommand{\bR}{\mathbb{R}}

\newcommand{\bZ}{\mathbb{Z}}

\newcommand{\dR}{\mathds{R}}

\newcommand{\wb}{w^{(b)}}
\newcommand{\ellb}{\ell^{(b)}}
\DeclareMathOperator{\E}{\mathds{E}}
\DeclareMathOperator{\pr}{\mathds{P}}

\DeclareMathOperator{\var}{Var}

\newcommand{\bes}{\begin{equation*}}
\newcommand{\ees}{\end{equation*}}
\newcommand{\be}{\begin{equation}}
\newcommand{\ee}{\end{equation}}

\begin{document}



\author[Dey]{Partha S.~Dey}
\address{Partha Dey, Department of Mathematics, University of Illinois at Urbana-Champaign, 1409 W.~Green St, Urbana, IL 61801, U.S.A.,  email: {\tt psdey@illinois.edu }}

\author[Joseph]{Mathew Joseph}
\address{Mathew Joseph, Indian Statistical Institute Bangalore, 8th Mile Mysore Road, RVCE Post, Bengaluru 560059, email: {\tt m.joseph@isibang.ac.in}}

\author[Peled]{Ron Peled}
\address{Ron Peled, 231 Schreiber Building, School of Mathematical Sciences, Tel Aviv University, Ramat Aviv, Tel Aviv 69978,
Israel, email: {\tt peledron@post.tau.ac.il }}

\title{Longest increasing path within the critical strip}

\begin{abstract}
A Poisson point process of unit intensity is placed in the square $[0,n]^2$. An increasing path is a curve connecting $(0,0)$ with $(n,n)$ which is non-decreasing in each coordinate. Its length is the number of points of the Poisson process which it passes through. Baik, Deift and Johansson proved that the maximal length of an increasing path has expectation $2n-n^{1/3}(c_1+o(1))$, variance $n^{2/3}(c_2+o(1))$ and that it converges to the Tracy-Widom distribution after suitable scaling. Johansson further showed that all maximal paths have a displacement of $n^{\frac23+o(1)}$ from the diagonal with probability tending to one as $n\to \infty$. Here we prove that the maximal length of an increasing path restricted to lie within a strip of width $n^{\gamma},  \gamma<\frac23$, around the diagonal has expectation $2n-n^{1-\gamma+o(1)}$, variance $n^{1 - \frac{\gamma}{2}+o(1)}$ and that it converges to the Gaussian distribution after suitable scaling.
\end{abstract}

\maketitle

\section{Introduction}\label{sec:intro}
The problem of determining the distribution of the length $\tilde{L}_n$ of the longest increasing subsequence in a random uniform permutation of $\{1,2,\cdots, n\}$ was first posed by Ulam and considered by Hammersley in his seminal paper~\cite{hamm}.  It has since attracted a lot of attention in the mathematical community; see Romik's book~\cite{romi} for a lively history of this problem and the remarkable progress that has been made since it was first introduced. Using the subadditive ergodic theorem,  Hammersley~\cite{hamm} was able to show the existence of a constant $c$ such that $ \tilde{L}_n/ \sqrt n \to c$ in probability as $n \to \infty$. A few years later, Logan and Shepp~\cite{loga-shep-77}, and Vershik and Kerov~\cite{vers-kero-77} were able to compute the exact value $c=2$ using an analysis of random Young tableaux. A different proof was provided by Aldous and Diaconis~\cite{aldo-diac95} by studying of the hydrodynamical limit of an interacting particle system implicit in the work of Hammersley, which is now called the Hammersley process in his honor.

The above results can be considered as a law of large numbers for $\tilde L_n$. The limiting distribution for $\tilde L_n$, under appropriate centering and scaling, was found to be the Tracy-Widom distribution in the landmark paper by Baik, Deift and Johansson \cite{baik-deif-joha-99}. The paper is a tour-de-force using the Robinson-Schensted-Knuth correspondence between permutations and Young tableaux, and elements of Riemann-Hilbert theory. It is convenient to work with the following equivalent formulation of the problem, as done already by Hammersley~\cite{hamm} and as we will do in the rest of the paper. Consider a Poisson point process of unit intensity in the square $[0,n]^2$. Let $P_n$ be the (random) number of points inside the square. An increasing path is a curve connecting $(0,0)$ with $(n,n)$ which is non-decreasing in each coordinate. Its \emph{length} is the number of points of the Poisson process which it passes through. Denote by $L_n$ the length of the longest increasing path. The relation between $L_n$ and $\tilde{L}_n$ is given by the fact, which is straightforward to check directly, that conditioned on $P_n = k$, the distribution of $L_n$ equals the distribution of $\tilde L_{k}$. One can therefore obtain the limiting distribution of $\tilde L_{P_n}$ by studying the {\it Poissonized} problem, and obtain similar results for $\tilde L_n$ via a de-Poissonization argument. Among the key findings of~\cite{baik-deif-joha-99} is that, as $n\to\infty$,
\begin{align}
  &\E L_n = 2n - n^{1/3}(c_1 + o(1)),\label{eq:unrestricted_exp}\\
  &\var L_n = n^{2/3}(c_2 + o(1)),\label{eq:unrestricted_var}\\
  &\frac{L_n - \E L_n}{\sqrt{\var L_n}}\text{ converges in distribution to the Tracy-Widom distribution}
\end{align}
with $c_1,c_2>0$ absolute constants.

In a subsequent development, Johansson \cite{joha-ptrf-00} showed using a geometric argument that the {\it transversal exponent} of the longest increasing path is $\frac 23$. More precisely, as the longest increasing path in $[0,n]^2$ need not be unique, it is shown there that with probability tending to $1$ as $n\to \infty$ all the longest increasing paths have a displacement of order $n^{\frac23 +o(1)}$ around the diagonal. The result raises the question of studying the maximal length attainable for increasing paths restricted to lie closer to the diagonal. Addressing this question is the main purpose of this paper.

Let $0< \gamma<\frac{2}{3}$ and set $L_n^{(\gamma)}$ to be the maximal length of an increasing path restricted to lie in $[0,n]^2\cap\{(x,y)\colon |y-x|\le n^{\gamma}\}$. As $\gamma$ decreases, the length $L_n^{(\gamma)}$ corresponds to a maximum over a smaller set. It is then clear that $L_n^{(\gamma)}$ decreases and one may further
\begin{theorem}\label{thm:strip} Fix $0< \gamma<\frac23$.  We have the following asymptotic behavior as $n\to \infty$:
\be \begin{split} \label{eq:e:var}
\E L_n^{(\gamma)} &= 2n - n^{1-\gamma +o(1)}, \\
\var L_n^{(\gamma)} &= n^{1-\frac{\gamma}{2} +o(1)}
\end{split}\ee
and
\be \label{eq:gauss:lim}
\frac{L_n^{(\gamma)}-\E L_n^{(\gamma)}}{\sqrt{\var{L_n^{(\gamma)}}}} \Rightarrow N(0,1).
\ee
\end{theorem}
The Theorem is further refined by the results in Section~\ref{sec:results} below.

The Gaussian fluctuations of the length of the longest increasing path inside the strip $n^{\frac23-\epsilon}$ are remarkably  different from the Tracy-Widom fluctuations of the length of the longest path within the strip $n^{\frac23+\epsilon}$. A tantalizing problem would be to study the transition when paths within the critical strip, having width of order $n^{\frac23}$, are considered.

The heuristic for the proof is as follows. Partition the strip $[0,n]^2\cap\{(x,y)\colon |y-x|\le n^{\gamma}\}$ into $n^{1-\frac{3\gamma}{2}+o(1)}$ diagonally-aligned rectangles (which we shall call {\it blocks}) of length $n^{\frac{3\gamma}{2}+o(1)}$ (the strip is partitioned apart from initial and final portions which are small enough to be ignored; see Figure~\ref{fig:partition of strip for main theorem}). The length of the longest increasing path within each block obeys the scaling given by \eqref{eq:unrestricted_exp}, \eqref{eq:unrestricted_var} as the height of each block is larger than the $2/3$-power of its length, allowing for transversal fluctuations which are essentially unrestricted due to the result of \cite{joha-ptrf-00}. If the length of the longest increasing path (LIP) in the entire strip was  the sum of the lengths of the LIPs within each block, then we would have a sum of i.i.d. random variables and we could apply the Lindeberg-Feller theorem to get our result. However, this is not true -- the restriction of a LIP (in the entire strip) to a block need not be maximal within this block. A key ingredient in our proof is Theorem \ref{thm:Om}, as a consequence of which we are able to show the existence of many blocks with {\it regeneration points} (these are the points $\mathbf{p}$ in \eqref{eq:Om}). A longest increasing path within the entire strip is obtained by concatenating the longest increasing paths between successive regeneration points. The regeneration points provide the necessary independence structure required for a Gaussian limit.

A model related to the above is directed last passage percolation on the two-dimensional integer lattice. One starts with a collection $\{Y_x\}_{x\in \bZ_{\ge 0}^2}$ of i.i.d. random variables, and is interested in studying $G_{m,n}= \max_{\pi} \sum_{x\in \pi} Y_x$, where the maximum is over all up-right paths starting from the origin and terminating at $(m,n)$. A law of large numbers and a Tracy-Widom limit for $G_{n,n}$ have been obtained for the case of exponential and geometric weights $Y_x$ (\cite{joha}, see the lecture notes \cite{sepp-lect}); the memory-less property of these distributions is used crucially. The model with exponential weights is connected to the totally asymmetric simple exclusion process (TASEP) with step initial condition (a particle on every $x \in \bZ_{<0}$). The random variable $G_{m,n}$ has the same distribution as the time it takes for the $m$th particle to move $n$ places to the right (see \cite{sepp-lect}). One might ask if a result similar to Theorem \ref{thm:strip} holds for $G_{n,n}^{(\gamma)}$, where we change the definition of $G_{n,n}$ so that $\pi$ lies within a strip of width $n^{\gamma}$ around the diagonal. While we have not attempted to do so, our arguments rely mainly on moderate deviation bounds of the type given in Lemma \ref{lem:E:Var} and  Lemma \ref{lem:L:tail} and may adapt for other models where similar bounds are known.

Let us also mention the following equivalent formulation of $G_{m,n}^{(\gamma)}$. Consider TASEP with particles to the left of the origin and a source with infinitely many particles at $-[n^{\gamma}]$. Particles follow the same rules as ordinary TASEP except that now the particles at the source jump (in order) to the site $-[n^{\gamma}]+1$ after it becomes vacant. There is a sink at $[n^{\gamma}]$ in which all particles eventually fall. It can then be argued that the distribution of $G_{m,n}^{(\gamma)}$ is the same as that of the time it takes for the $m$th particle (which might now start from the source) to make $n$ steps to the right.

The literature on first/last passage percolation has several results of relevance to this work. For (undirected) first passage percolation on $\bZ^d$ with weights $Y_x$ having all moments, \cite{chatt-dey} proved a Gaussian  central limit theorem for the first passage time when the paths are restricted to thin rectangles of the form $[-h_n,h_n]\times[-n,n]^{d-1}$ where $h_n =o(n^{1/(d+1)})$.  Under natural but unproven assumptions they are able to extend their result up to $h_n=n^{\xi'}$ where $\xi'<\xi$, the transversal exponent (see \cite{chat}, \cite{auff-damr}).  For directed last passage percolation on $\bZ_{\ge 0}^2$ and for $Y_x$ having all moments, one gets a Tracy-Widom limit for the last passage time for rectangles of the form $[0,n]\times [0,n^a]$ where $a<\frac37$ (\cite{baik-suid}, \cite{suid}, \cite{bodi-mart}). This is remarkably different from the Gaussian behavior of first passage percolation. In our present set-up, it is easy to see by a scaling argument that the behavior of the LIP in any rectangle with sides parallel to the axes is the same as the behavior of the LIP in a square of the same area,  and we would have a Tracy-Widom limit. With similar scaling, Theorem \ref{thm:strip} implies a Gaussian limit for the LIP in any inclined rectangle if its width is sufficiently smaller than its length. Another result which deals with the length of the LIP in non-square domains is that of \cite{baik-rain}. They consider the length of the LIP in a right-angled triangle with extra points distributed uniformly on the hypotenuse.  Using non-geometric arguments that link the model with random matrix theory, they are able to show different limiting behaviors depending on the number of points on the diagonal. Lastly, we mention a recent result of \cite{basu-sly-sido} where they prove a law of large numbers for the length of the LIP  in a square with points from a Poisson point process of intensity $\lambda$ added to the diagonal. Using geometric arguments of a similar flavor as ours, they are able to show that the limiting constant is strictly greater than $2$ for any $\lambda>0$. The analogue of this for the last passage percolation model is called the {\it slow bond} problem.

As an interesting aside, we note that the study of the longest increasing subsequence in a restricted geometry also arises in application areas; see~\cite{bachmat2006analysis}, \cite[Chapter 3]{bachmat2014mathematical} for an application to airplane boarding times.

\subsection{Notation and Main Results}\label{sec:results}

We consider a Poisson process of unit intensity on $\dR^2$.
Instead of working with the usual $(x,y)$ coordinates in $\dR^2$, we
shall use diagonal coordinates $(t,s)$. Here $t$ measures the
distance along the line $y=x$ and $s$ measures the distance along
the line $y=-x$. The correspondence between the coordinate systems is
\begin{equation*}
  \begin{split}
    &x = \frac{t-s}{\sqrt 2},\quad y = \frac{t+s}{\sqrt 2}\qquad\text{ and } \qquad
    t = \frac{x+y}{\sqrt 2},\quad s = \frac{y-x}{\sqrt 2}.
  \end{split}
\end{equation*}
Thus a diagonal rectangle $R$ of length $t_0$ and
width $s_0$ with lower-left corner at $(0,0)$ is the region
$\{(t,s)\colon 0\le t\le t_0,\, 0\le s\le s_0\}$ (see Figure~\ref{fig:coord}). From now on we will use always the $(t,s)$ coordinate system and thus an increasing path will refer to a path in the $t$--$s$ plane whose slope at every point is in $[-1,1]$ (in the sense that $\frac{s_2 - s_1}{t_2-t_1}\in[-1,1]$ for pairs of distinct points $(t_1, s_1)$, $(t_2, s_2)$ along the path).
\begin{figure}[htbp]
   \centering
   \includegraphics[width=250pt]{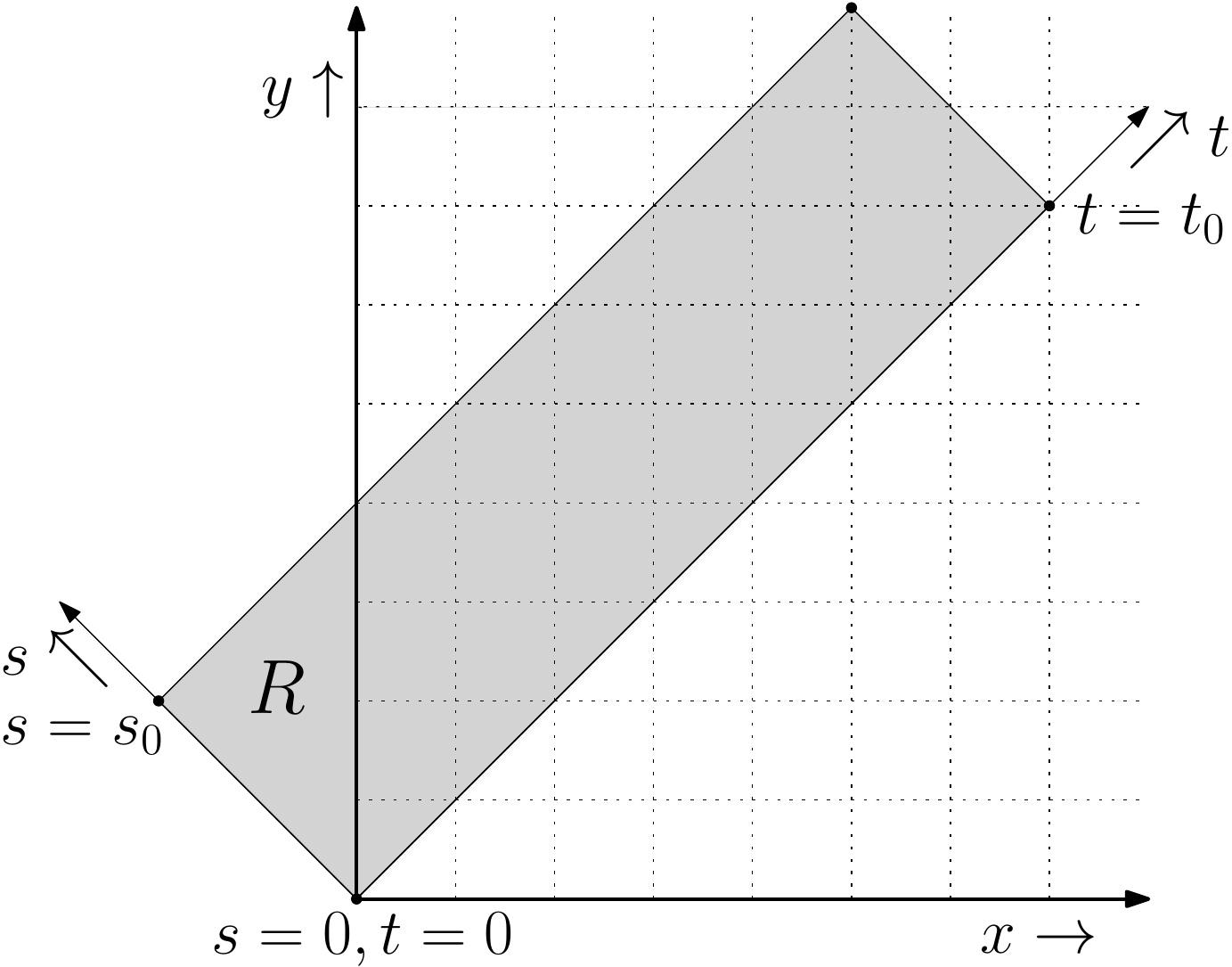}
      \caption{Diagonal coordinate system with diagonal rectangle $R$}
   \label{fig:coord}
\end{figure}

Denote by $\cL\big( (t_1,s_1), (t_2,s_2)\big)$ the collection of LIPs between points $(t_1,s_1)$ and $(t_2,s_2)$ in the $t$--$s$ plane. The length of, that is  the number of points on, any LIP in this collection is denoted by $L\big((t_1,s_1),(t_2,s_2)\big)$ and its expectation by $E\big((t_1,s_1),(t_2,s_2)\big)$. For a fixed domain $R$ (e.g.~a diagonal rectangle), $\cL^R\big((t_1,s_1),(t_2,s_2)\big)$ will denote the collection of LIPs between $(t_1,s_1)$ and $(t_2,s_2)$ constrained to lie within the domain $R$. We will use the notation $L^R\big((t_1,s_1),(t_2,s_2)\big)$ for the corresponding length and $E^R\big((t_1,s_1),(t_2,s_2)\big)$ for the expectation of this length. We write $L^R$ to denote the length of the LIP in $R$ when allowing arbitrary initial and end points in $R$.
We shall use the shorthand
\bes
\begin{split}
L_t &:= L\big((0,0),(t,0)\big) \qquad\text{ and }\qquad
L_{t,s} := L \big((0,0),(t,s)\big).
\end{split}
\ees
Note that $L_{t,s}=0$ if $|s|>t$ and $L_{t,s}$ equals $L_{\sqrt{t^2-s^2}}$ in distribution when $|s|\le t$, since the area of the rectangle determines the distribution of the length of the LIP (by applying a linear transformation to the Poisson process).
Correspondingly, we set
\bes\begin{split}
E_t:= \E L_t\qquad & \text{and} \qquad E_{t,s} := \E L_{t,s}\,,  \\
\sigma_t:= \sqrt{\var(L_t)} \qquad &  \text{and} \qquad \sigma_{t,s}:=\sqrt{\var(L_{t,s})}\,.
\end{split}
\ees

Consider now a diagonal rectangle $R$ of length $\ell$ and width $w$. The following theorem gives bounds on the expectation of $L^R$. The first assertion in \eqref{eq:e:var} follows as a consequence.
\begin{theorem}\label{thm:exp} Fix $0<\delta< 1/6$. Consider a diagonal rectangle $R$ of length $\ell$ and width $w$ such that $w(\log w)^{\frac53} \le\ell^{\frac23}$. There exist constants $c(\delta), \, C(\delta)$ depending only on $\delta$ such that for $w$ large enough
\begin{align}\label{eq:e:lr}
\sqrt 2 \ell -C(\delta)\cdot  (\log w)^{\frac13+\delta}\cdot \frac{\ell }{w} \le \E L^R \le \sqrt 2 \ell - \frac{c(\delta)}{  (\log w)^{\frac53-\delta} }\cdot \frac{\ell}{w}
\end{align}
\end{theorem}
The first term $\sqrt{2} \ell$ matches the leading term in the expectation of the length of the unrestricted LIP in square domains. Our next result is a lower bound on the variance of $L^R$.
\begin{theorem} \label{thm:var:lbd} Consider a diagonal rectangle $R$ of length $\ell$ and width $w$ with
$\frac{w}{\log \log w}\le \ell^{\frac23}$. There exists an $\epsilon_0>0$ such that for any $0<\epsilon<\epsilon_0$,
we have

 \bes
\var(L^R) \ge \left(\frac{\ell}{w^{\frac12}}\right)^{1-\eps}
\ees
for $w$ large enough.
\end{theorem}
Complementing this result, we provide upper bounds for all centered moments of $L^R$. These two theorems together prove the second assertion in \eqref{eq:e:var}.
\begin{theorem} \label{thm:mom:ubd} Consider a diagonal rectangle $R$ of length $\ell$ and width $w$ such that
$\frac{w}{(\log w)^{1/2}}\le\ell^{\frac23}$. For each fixed $k\ge 1$,  there exists a constant $C(k)$ depending only on $k$ such that for $w$ large enough
\bes
\left\| L^R - \E L^R\right\|_{2k} \le C(k)\cdot \frac{\ell^{\frac12}}{w^{\frac14}}\cdot (\log w)^{1+k},
\ees
where we use the notation $\| X\|_p := (\E X^p)^{1/p}$.
\end{theorem}
The theorems above suggest that the variance decreases as the width increases for each fixed~$\ell$. It is worth noting that such a monotonicity property was a key (unproven) assumption made in the paper \cite{chatt-dey} for first passage percolation. Our final theorem is a Gaussian limit theorem for $L^R$ when the width is less than the critical width.
\begin{theorem}\label{thm:clt}
Fix $\gamma<\frac{2}{3}$. Consider a sequence of diagonal rectangles $R = R(\ell)$ of length $\ell$ and width $w(\ell)$ such that $w(\ell)\le \ell^\gamma$.
Then, as $\ell\to \infty$:
\begin{equation} \label{eq:clt}
\frac{L^R-\E L^R}{\sqrt{\var\, L^R}} \Rightarrow N(0,1).
\end{equation}
\end{theorem}
\begin{remark} In the above results we have considered $L^R$, the length of the LIP within the rectangle $R$. One can also obtain similar results for $L^R(\mathbf{a},\mathbf{b})$, the length of the LIP among paths starting at $\mathbf{a}$, ending at $\mathbf{b}$ and restricted to lie within $R$, where $\mathbf{a},\mathbf{b}$ are points on the left and right boundaries of the rectangle $R$. This can be done by arguing as in the proof of Theorem \ref{thm:strip} given in Section \ref{sec:CLT}.
\end{remark}
\begin{remark} The main ingredients in our results are the tail estimates in Lemma \ref{lem:L:tail} and Lemma \ref{lem:tr_fl}. We believe that results of a similar flavor hold for other models where corresponding estimates are known.
\end{remark}
We now describe the outline of the paper. Section \ref{sec:prelim} describes the moderate deviation results as well as the van-den-Berg-Kesten type inequality that our analysis relies upon, and collects several consequences. Section \ref{sec:exp:lbd} gives a lower bound on the expectation of the length of the LIP in diagonal rectangles and Section \ref{sec:up:bd} provides error estimates for the length of the LIP in {\it short blocks}. We next prove Theorem \ref{thm:exp} in Section \ref{sec:exp}. Section \ref{sec:merge} contains the main argument of the paper in which we give a lower bound on the probability that LIPs (between arbitrary starting and ending points on the left and right boundaries respectively) in each block meet at a point. We use this result to prove Theorem \ref{thm:var:lbd} in Section \ref{sec:var:lbd}. After this we prove Theorem \ref{thm:mom:ubd} in Section \ref{sec:mom:ubd} and then Theorem \ref{thm:clt} in Section \ref{sec:CLT}. We end the paper with a few open questions for the reader.

Throughout the article, $c,C$ will be used to denote arbitrary positive constants which may change from line to line. If the constants depend
on a certain parameter this shall be indicated in parenthesis.

\subsection*{A remark on timing}
A preliminary version of this paper containing full statements and proofs of the above results is available since December 2015 on the website of the first author. The second author presented these results at seminars in Bristol, Oxford and Sussex in 2015--16 while the third author discussed these at the June 2015 ``Groups, Graphs and Stochastic Processes'' Banff workshop\footnote{A video is available at \url{http://www.birs.ca/events/2015/5-day-workshops/15w5146/videos/watch/201506231638-Peled.html}}. For various reasons, polishing of the draft and its final posting to the arXiv were greatly delayed, until its final appearance there on August 2018. In the interim, many results on last passage percolation, in the continuum and on the lattice have appeared, including \cite{basu-sark-sly,basu-gang-sly,basu-gang-hamm,basu2018time,hammond2016brownian,hammond2017modulus,hammond2017rarity, hammond2017patchwork, hamm-sark}. Among these, especially related to the present work is a recent result of Basu, Ganguly and Hammond \cite{basu-gang-hamm} who consider the length of the longest increasing path between $(0,0)$ and $(n,n)$, constrained to enclose an atypically large area of at least $(1/2+\alpha)n^2$. Apart from finding the law of large numbers for the model (in fact, they also find a limiting curve), the key result of their work is that the fluctuations around the limiting curve are of order $n^{1/2}$ and the convex hull facet length is of order $n^{3/4}$. This should be contrasted with the orders $n^{1/3}$ and $n^{2/3}$ for the unconstrained model. In a second related work which appeared very recently, Basu and Ganguly \cite{basu2018time} find, among other results, upper and lower bounds on the variance of the last passage percolation time in thin cylinders which match up to a multiplicative constant. The latter two works kindly acknowledged our preprint.

\section{Preliminaries}\label{sec:prelim}
In this section we collect several results which we make use of in the
sequel.
\subsection{Asymptotic Behavior}
The following was proved in \cite{baik-deif-joha-99}.
\begin{lemma}\label{lem:E:Var} There exist positive constants $c_0, c_1$ such that
\begin{align}
\sigma_t &    = c_0 t^{1/3}\big(1+o(1)\big) \label{eq:sig:asy}, \\
 E_t & = \sqrt{2}\, t - c_1 t^{1/3} \big(1+o(1)\big) \label{eq:E:asy}
\end{align}
as $t\to \infty$. We also have for all $k \ge 1$,
\begin{equation} \label{eq:mom:bd}
 \big\|L_t- \sqrt{2}\, t \big \|_k \le C t^{1/3}.
\end{equation}
for some constant $C>0$.
\end{lemma}
The next lemma gives tail bounds on the deviation of the length of the LIP
from its mean. This lemma is the black box for our arguments and shall be referred to
several times during the course of the paper. It collects results proved in \cite{baik-deif-joha-99},
\cite{lowe-merk} and \cite{lowe-merk-roll}.
\begin{lemma}\label{lem:L:tail}  The following hold.
\begin{enumerate}
\item (Upper tail bounds) There exist constants $c, C>0$ such that for sufficiently large $t$,
\begin{align}\label{eq:up:tail}
c \exp\big(-CT^{3/2}\big)\le \pr \left(\frac{L_t- E_t}{\sigma_t} \ge T\right) \le C \exp\big(-c T^{3/2}\big)
\end{align}
for all $1\le T\le t^{2/3}$. \\
\item (Lower tail bound) There exist constants $c, C>0$ such that for sufficiently large $t$,
\begin{align}\label{eq:low:tail}
\pr \left(\frac{L_t- E_t}{\sigma_t} \le -T\right) \le C\exp\big(-c T^{3}\big)
\end{align}
for all $1\le T\le t^{2/3}$.
\end{enumerate}
\end{lemma}
\begin{proof}
The upper bounds in \eqref{eq:up:tail} and \eqref{eq:low:tail} have been
proved in Lemma 7.1 of \cite{baik-deif-joha-99}; see also equations (1.4) and (1.5) in \cite{joha-ptrf-00}.
We now prove the lower bound in \eqref{eq:up:tail}. From Lemma~\ref{lem:E:Var} and super-additivity of $E_{r}$, we have that
\[
0\le \frac{\sqrt{2}r-E_{r}}{\sigma_{r}}\le a \text{ and } b_1\le \frac{\sigma_{r}}{r^{1/3}}\le b_2 \text{ for all } r\ge 1
\]
for some universal constants $a,b_{1},b_{2}>0$.  From the distributional convergence of $2^{1/6}t^{-1/3}(L_t-\sqrt{2} t)$ to the Tracy-Widom
disrtibution (see \cite{baik-deif-joha-99}) we note that there
exists a number $\gl>0$ such that
\bes
\pr(L_q-E_q\ge k \sigma_q) \ge e^{-\gl} \quad \mbox{ for all } q\ge 1.
\ees
Let $k=\lceil a+{b_2}/{b_1}\rceil$ and choose $r= t T^{-3/2}\ge 1$ and now consider the event
\bes
A=\bigcap_{j=0}^{\lceil T^{3/2}\rceil} \Big\lbrace L\big((jr,0), ((j+1)r,0)\big) -E_r \ge k\sigma_r\Big\rbrace
\ees
The events in the braces are all independent and hence $\pr(A)\ge e^{-\gl( \lceil T^{3/2}\rceil +1)}$. Next note that on the event $A$, we have
\bes \begin{split} L_t
& \ge \sum_{j=0}^{\lceil T^{3/2}\rceil}  L\big((jr,0), ((j+1)r,0)\big) \\
& \ge  k T^{3/2} \sigma_r + T^{3/2} E_r \\
&\ge  \Bigl(\frac{\sqrt{2}r-E_r}{\sigma_r} + \frac{b_2}{b_1}\Bigr)T^{3/2} \sigma_r+ T^{3/2} E_r \\
&\ge  \sqrt{2}rT^{3/2} + \frac{\sigma_t}{t^{1/3}}\cdot\frac{r^{1/3}}{\sigma_r}\cdot T^{3/2}\sigma_r
=\sqrt{2}t +T\sigma_t
\ge E_t+T\sigma_t.
 \end{split} \ees
 This completes the proof of the lemma.
\end{proof}
The following fundamental lemma is a consequence of Lemma \ref{lem:E:Var}.
\begin{lemma} \label{lem:s:t} There exists a small $\alpha_0>0$ such that the following hold.
\begin{enumerate}
\item There exist positive constants $c,C$ such that whenever $|s|/t\le \alpha_0$
\begin{align}\label{eq:E:diff}
\frac{cs^2}{t}\le E_t - E_{t,s} \le \frac{Cs^2}{t}
\end{align}
\item There exist constants $c_2,c_3> 0$ such that whenever $|s|/t\le \alpha_0$
\bes
c_2 t^{1/3} \le \sigma_{t,s} \le c_3 t^{1/3}
\ees
when $t$ is large enough. The upper bound holds for all $|s|\le t$.
\item The tail bounds in Lemma \ref{lem:L:tail} hold (perhaps with different constants)
with $L_{t}, E_t, \sigma_t$ replaced
by $L_{t,s}, E_{t,s}, \sigma_{t,s}$ respectively.
\end{enumerate}
\end{lemma}
\begin{proof} The proof is a ready consequence of the distributional identity $L_{t,s}\stackrel{d}{=} L_{\sqrt{t^2-s^2}}$ and the fact that $\sqrt{t^2-s^2} = t- \frac{s^2}{2t}(1+O(\frac{s^2}{t^2}))$ when $|s|\ll t$.
\end{proof}

\subsection{Transversal Fluctuations}
In this section, we shall give a quantitative control on the transversal fluctuations
of the LIP in square domains $[0,n]^2$. For this we shall discretize the space
 to a lattice of points of mesh size $1$. In
order to show that the lengths of LIPs between points of the lattice are good approximations of
the lengths of LIPs between general points we need a control over the number of Poisson
points in small squares. Thus we define the event
\begin{equation*}
  A_0:=\left\{\text{The number of points in each $1\times 1$ square
  contained in $[0,n]^2$ is less than $(\log n)^2$} \right\}.
\end{equation*}
\begin{lemma} \label{lem:grid}
  We have
  \begin{equation*}
    \pr\left(A_0^c\right) \le C\exp\left(-c(\log n)^2 \log \log n\right).
  \end{equation*}
\end{lemma}
\begin{proof}
  It suffices to cover $[0,n]^2$ by $n^2$ squares of side-length
  $1$ and show that no such square contains more than $(\log n)^2 / 4$
  Poisson points with high probability. This is because every square of
  side-length $1$ intersects at most $4$ of the $n^2$ squares. The
  number of Poisson points in a square of side-length $1$ is
  Poisson$(1)$ and hence the probability for it to be larger than
  $k$ is bounded by $\exp(-k\log k+k-1)$, yielding the lemma.
\end{proof}
Let $S$ denote the maximal absolute value of the $s$ coordinate of
  a point on a maximizing path connecting $(0,0)$ to $(n,0)$ (in the
  $t$--$s$ plane). More precisely,
  \bes
  S= \max \left\lbrace |\omega(t)|: 0\le t\le n,\, t \mapsto \omega(t) \text{ is a maximizing path in } \cL\big((0,0),(n,0)\big)\right\rbrace
  \ees
Thus $S$ measures the transversal fluctuations of
  the path.
\begin{lemma}(Transversal fluctuations upper bound)\label{lem:tr_fl}
 For every $0<\delta<\frac{1}{3}$ there exits a positive integer $n_0(\delta)$, such that for $n\ge n_0(\delta)$ we have
  \begin{equation*}
    \pr\left(S > n^{2/3}\left(\log n\right)^{\frac{1}{3}+\delta}\right)\le
    C\exp\big(-c(\log n)^{1 + 3\delta}\big).
  \end{equation*}
\end{lemma}

\begin{proof}
Assume the event $A_0$ occurs as the probability of $A_0^c$ is much
smaller than the probability we are interested in. Put a lattice of
side length $1$ inside the square with corner points $(0,0)$ and
$(n,0)$. If $S > n^{2/3}\left(\log n\right)^{\frac{1}{3}+\delta}$ there
exists a point $(t,s)$ of the lattice with $\big\vert |s|-n^{2/3}\left(\log
n\right)^{\frac{1}{3}+\delta}\big\vert \le 1$ such that
\begin{equation*}
  L((0,0), (n,0)) \le L((0,0), (t,s)) + L((t,s), (n,0)) +
  (\log n)^2.
\end{equation*}
However, for each such $(t,s)$
 \begin{equation*}\begin{split}
  E_n - E_{t,s} - E_{n - t, -s}  & = (E_n-E_t-E_{n-t}) + (E_t-E_{t,s}) + (E_{n-t}-E_{n-t,-s})
  \end{split}\end{equation*}
The first term in parenthesis on the right is of order $n^{1/3}$. The second and third terms in paranthesis are clearly nonnegative. When $|s|/t\le \alpha_0$ (the constant appearing in Lemma \ref{lem:s:t}) we have by \eqref{eq:E:diff} that $E_t - E_{t,s} \ge \frac{cs^2}{t}$. On the other hand when $|s|/t > \alpha_0$ we have $|s|/(n-t) \le\alpha_0$ and therefore $E_{n-t}-E_{n-t,-s} \ge \frac{cs^2}{n-t}$. In either case we obtain
\be
 E_n - E_{t,s} - E_{n - t, -s} \ge cn^{1/3}(\log n)^{2/3+2\delta}
\ee
since $\big\vert |s|-n^{2/3}\left(\log
n\right)^{\frac{1}{3}+\delta}\big\vert \le 1$. The standard deviations of $\sigma_n, \sigma_{t,s}$ and $\sigma_{n-t,-s}$
are all of order at most $n^{1/3}$. Thus, by
\eqref{eq:up:tail}, \eqref{eq:low:tail} and similar bounds in Lemma \ref{lem:s:t} it follows
that
\bes
\begin{split}
& \pr\left(  L((0,0), (n,0)) \le L((0,0), (t,s)) + L((t,s), (n,0)) +
  (\log n)^2\right) \\
& \le \pr\left(| L_{t,s} - E_{t,s} | >c n^{1/3}(\log n)^{2/3+2\delta} \right) +  \pr\left(| L_{n-t,-s} - E_{n-t,-s} | >c n^{1/3}(\log n)^{2/3+2\delta} \right) \\
& \qquad  \qquad \qquad+  \pr\left(| L_n- E_{n} | >c n^{1/3}(\log n)^{2/3+2\delta} \right) \\
 &\le  C\exp\big(-c(\log n)^{1 + 3\delta}\big).
\end{split}
\ees
 Therefore the probability that a maximal path from $(0,0)$ to $(n,0)$
passes at distance $1$ from the lattice point $(t,s)$ is at most $C\exp(-c(\log
n)^{1+3\delta})$. As there are less than $Cn^2$ points in our
lattice, the lemma follows.
\end{proof}
The following definition will be useful for us. It simplifies the exposition considerably.
\begin{definition} We say an event $A$ depending on the points in $\big[0,n]^2$
occurs with ``overwhelming probability" if there is a $\theta>0$ such that
\bes
\pr(A^c) \le \exp\big(-c (\log n)^{1+\theta}\big).
\ees
\end{definition}
In particular the event $\{S\le n^{\frac23}(\log n)^{\frac13+\delta}\}$ occurs with overwhelming probability.
A consequence of the above definition which we shall use several times without
explicitly referring to is
\begin{lemma} Let the event $A$ depending on the points in $\big[0,n\big]^2$ occur with
overwhelming probability. Then for every $\eps>0$ and $k\ge 1$
\bes
\big \| L_{\sqrt 2 n}  \cdot \ind_{A^c} \big\|_k  \le n^\eps
\ees
for $n$ large enough.
\end{lemma}
\begin{proof} Let $K_n$ be the number of Poisson points in $[0,n]^2$. The following
follows from a union bound by splitting the square into squares of side length $1$.
\bes \begin{split}
\pr\big(L_{\sqrt 2n} \ge m\big) & \le \pr\big( K_n \ge m\big) \\
& \le n^2 \pr \Big(\text{Poisson}(1)> \frac{m}{n^2}\Big) \le \exp\Big(-\frac{m}{n^2}\cdot\log \left(\frac{m}{n^2}\right)\Big)
\end{split}\ees
Therefore
\bes \begin{split}
\E\big(L_{\sqrt 2n}^k \cdot \ind_{A^c}\big) & \le \sum_{m \le n^3} n^{3k} \E\big(\ind_{A^c}\cdot \ind\{L_{\sqrt 2 n}\le n^3\}\big) \\
&\qquad+ \sum_{i=1}^{\infty} n^{3(i+1)k}  \E\big(\ind_{A^c}\cdot \ind\{in^3\le L_{\sqrt 2 n}\le (i+1)n^3\}\big)
\end{split}\ees
which is smaller than $n^\eps$ when $n$ is large enough, since $A$ occurs with overwhelming probability and we have a fast decay in the tail probabilities of $L_{\sqrt2 n}$ from Lemma \ref{lem:L:tail}. We will leave the details of this
computation to the interested reader.
\end{proof}

\subsection{BK inequality} We now state a generalization of the classical BK inequality in site percolation suitable for our case. Denote by $\Omega$ the collection of all realizations $\omega$ of the Poisson point process on $\bR_+^2$. We denote $\omega \subseteq \omega'$ if all the Poisson points in $\omega$ are also present in $\omega'$ and $\omega' \backslash \omega$ will then be the realization with all the points in $\omega$ removed. For a subset $K \subseteq \bR_+^2$ we denote $\omega_K$ to be the realization containing only the points of $\omega$ in $K$.  A set $A \subseteq \Omega$ is increasing if $\omega \in A$ implies $\omega' \in A$ when $\omega \subseteq \omega'$. For increasing sets $A$ and $B$ define the disjoint occurence
 \bes \begin{split}
A \Box B &=  \{\omega: \text{there exist disjoint regions } K,\, L\subseteq \bR_+^2 \text{ such that } \omega_K \in A \text{ and } \omega_L \in B \} \\
&= \{\omega: \text{there exists } \omega' \subseteq \omega \text{ such that } \omega' \in A \text{ and } \omega \backslash \omega' \in B\}.
\end{split}\ees
We then have
\begin{lemma}\cite{berg} \label{lem:bk}
For increasing events $A, B \subseteq \Omega$,
\bes
\pr(A\Box B) \le \pr(A)\cdot \pr(B).
\ees
\end{lemma}

\section{Lower Bound of Expectation in Blocks} \label{sec:exp:lbd}
Consider the diagonal rectangle $R$ of length $\ell$
and width $w$. Consider two points $(0,a)$ and $(\ell,b)$
on opposite sides of the rectangle. Recall that $E^R\big[(0,a),(\ell,b)\big]=\E L^R\big((0,a),(\ell,b)\big)$  is the expected value
of the restricted LIP between the points $(0,a)$ and $(\ell,b)$.
The following proposition shows that $E^R\big[(0,a),(\ell,b)\big]$ is not very far
from $E_{\ell} $.
\begin{proposition}\label{prop:ER:lbd}
Fix $0<\delta<\frac13$ and let $ \ell^{\frac23} (\log \ell)^{\frac13+\delta} \le w \le \alpha_0 \ell$ with
$\ell$ large enough and $\alpha_0$ as in Lemma \ref{lem:s:t}. There exists a constant
 $\mathcal{C}=\mathcal{C}(\delta)$ such that for any $0\le a, b\le w$,
\bes
E_{\ell}- \mathcal{C}\frac{w^2}{\ell} \le E^R\big[(0,a),(\ell,b)\big] \le E_{\ell}.
\ees
\end{proposition}
\begin{proof} The second inequality is immediate so let us consider the first inequality.
It is sufficient to show that
\bes
E^R\big[(0,a), \big(\ell/2,w/2\big)\big] \ge E_{\frac{\ell}{2}} - C \frac{w^2}{\ell}.
\ees
because by symmetry we have a similar relation for $E^R\big[ \big(\ell/2,w/2\big), \big(\ell,b\big)\big]$.
We shall prove this with $a=0$ and indicate at the end how to modify the argument for general $a$.
We shall consider paths which enter the interior of $R$ via successively
larger rectangles such that the paths in each sub rectangle behave just like
unrestricted paths. We make this precise as follows. Let
\bes
t_0=s_0=\ell^{1/3}.
\ees
We now define recursively two sequences
\bes\label{eq:sk:tk}
\begin{split}
s_{k+1}&= 2s_k, \\
t_{k+1}&= t_k +\frac{s_k^{3/2}\ell }{2w^{3/2}}.
\end{split}
\ees
Choose $k_0=O\big(\log_2 (w/\ell^{1/3})\big)$ so that $s_{k_0}\le \frac{w}{2} <s_{k_0+1}$. For this $k_0$ one can see that  $\frac{\ell}{12} \le t_{k_0} \le \frac{\ell}{3}$  because
\bes
t_{k_0} = \ell^{1/3} +\frac{\ell^{3/2}}{2 w^{3/2}}\left[1+2^{3/2}+2^{2\cdot 3/2}+\cdots+ 2^{(k_0-1)\cdot 3/2}\right]
\ees
\begin{figure}[htbp]
   \centering
   \includegraphics[width=350pt]{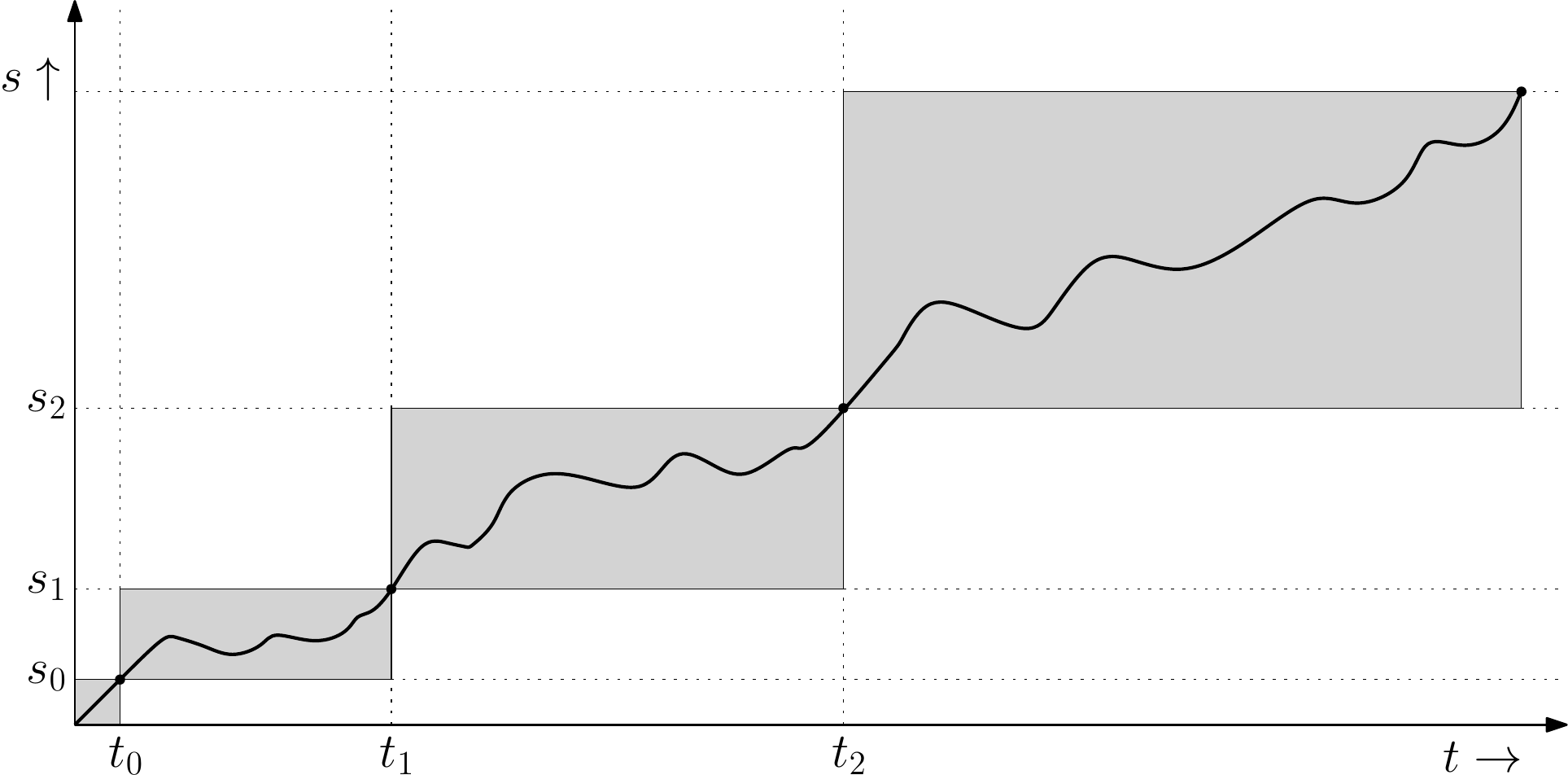}
   \caption{Construction of the restricted path.}
   \label{fig:path}
\end{figure}
We now create an increasing (restricted to $R$) path from $(0,0)$ to $(\ell/2,w/2)$ by first moving directly to $(t_0,s_0)$ (not passing through any Poisson points) and then for each $0\le k < k_0$, taking a LIP in $R$ from $(t_k,s_k)$ to $(t_{k+1},s_{k+1})$. By using the transversal fluctuation Lemma~\ref{lem:tr_fl}, as
    $t_{k+1} - t_k$ is short compared with $s_k$, we may treat each
    such LIP as an unrestricted LIP (not constrained to lie in $R$)
    as the unrestricted and restricted LIP coincide with overwhelming
    probability. In the last step we take a LIP in $R$ from $(t_{k_0},s_{k_0})$
    to $(\ell/2,w/2)$ which again coincides
    with the unrestricted LIP with overwhelming probability. In particular one can show that
    \be \label{eq:first}
    \big\vert E^R\big((t_k,s_k), (t_{k+1},s_{k+1})\big) - E\big((t_k,s_k), (t_{k+1},s_{k+1})\big) \big\vert \le \ell^{1/6}
    \ee
    and
    \be \label{eq:second}
     \Big\vert E^R\big((t_{k_0},s_{k_0}),(\ell/2,w/2)\big) - E\big((t_{k_0},s_{k_0}), (\ell/2,w/2)\big) \Big\vert \le \ell^{1/6}.
    \ee
 Let us show the first inequality \eqref{eq:first} above and leave the second \eqref{eq:second} to the reader.
 \bes
 \begin{split}
 &\big\vert E^R\big((t_k,s_k), (t_{k+1},s_{k+1})\big) - E\big((t_k,s_k), (t_{k+1},s_{k+1})\big) \big\vert  \\
 & \le 2 \E \left[L\big((t_k, s_k),\, (t_{k+1}, s_{k+1})\big) \cdot \ind \{\text{path goes out of the rectangle with endpoints } (t_k, s_k),\, (t_{k+1}, s_{k+1})\} \right] \\
 & \le 2\ell \cdot \pr\left(\{\text{path goes out of the rectangle with endpoints } (t_k, s_k),\, (t_{k+1}, s_{k+1})\}\right) \\
 &\qquad +\sum_{i=0}^\infty (2\ell + i \ell^{1/3})\cdot \pr\left(2\ell +i\ell^{1/3}\le L\big((t_k, s_k),\, (t_{k+1}, s_{k+1})\big)\le 2\ell +(i+1)\ell^{1/3}\right).
 \end{split}
 \ees
 Thanks to the tail bounds in Lemma \ref{lem:L:tail} and the bound on the probability of large transversal fluctuations in Lemma \ref{lem:tr_fl}, the inequality \eqref{eq:first} follows for all large $\ell$.

Since the maximal restricted path in $R$ has larger length than the length of the path obtained from our construction
we get, using Lemma \ref{lem:s:t}
\bes
\begin{split}
& E_{\frac{\ell}{2}} - E^R\big[(0,a), (\ell/2,w/2)\big] \\
&\quad  \le E_{\frac{\ell}{2}} -\sum_{k=0}^{k_0-1}E\big((t_k,s_k), (t_{k+1},s_{k+1})\big) - E\big((t_{k_0},s_{k_0}), (\ell/2,w/2)\big) + C\ell^{1/6} k_0\\
&\quad \le E_{\frac{\ell}{2}} -\sum_{k=0}^{k_0-1}\Big[E_{t_{k+1}-t_k} - C\frac{(s_{k+1}-s_k)^2}{t_{k+1}-t_k}\Big] - \Big[E_{\frac{\ell}{2}-t_{k_0}} - C\frac{(w/2)^2}{\frac{\ell}{2}-t_{k_0}}\Big]+ C\ell^{1/6} k_0 \\
&\quad \le E_{\frac{\ell}{2}} -\sum_{k=0}^{k_0-1} E_{t_{k+1}-t_k} -E_{\frac{\ell}{2}-t_{k_0}} + C\Big[\sum_{k=0}^{k_0-1}\frac{(s_{k+1}-s_k)^2}{t_{k+1}-t_k} + \frac{(w/2)^2}{\frac{\ell}{2}-t_{k_0}}\Big] + C\ell^{1/6} k_0 \\
&\quad \le C \ell^{1/3} + C \sum_{k=0}^{k_0-1} \Big\lbrace \frac{s_k^{3/2}\ell}{2 w^{3/2}}\Big\rbrace^{1/3} + C \Big[\sum_{k=0}^{k_0-1} \frac{s_k^2}{s_k^{3/2}\ell/2w^{3/2}}+ C\frac{w^2}{\ell}\Big]  + C \ell^{1/6}k_0\\
& \quad \le C \ell^{1/3}+ \frac{C\ell^{1/3}}{2^{1/3}w^{1/2}} \sum_{k=0}^{k_0-1} 2^{k/2} \ell^{1/6} + C \sum_{k=0}^{k_0-1} \frac{2^{k/2}\ell^{1/6}w^{3/2}}{\ell} + C \frac{w^2}{\ell} + C\ell^{1/6} k_0 \\
& \quad \le C \frac{w^2}{\ell}.
\end{split}
\ees
We have used the bound to $2^{k_0/2} \le w^{1/2}/ \ell^{1/6}$ to arrive at the final step.

Now we explain how the argument can be extended for any $a$.
For $a\le \ell^{1/3}$, we follow the same rectangles. For $\ell^{1/3}<a\le  w/2$ and $s_i < a\le s_{i+1}$, we take the
maximal path to $(t_{i+1}, s_{i+1})$ (which will remain in $R$ with overwhelming probability) and then follow the remaining
rectangles. The argument for $a\ge w/2$ follows along the same lines by symmetry.
\end{proof}

\section{Upper bounds in blocks} \label{sec:up:bd}
Consider a diagonal rectangle $R$ of width $w$ and length $\ell$. Let
\begin{align}\label{eq:Delta}
\Delta(R) := \max_{\vx\in B_l, \, \vy\in B_r} L^R(\vx,\vy) - \min_{\vx\in B_l, \, \vy\in B_r} L^R(\vx,\vy).
\end{align}
Here $B_l,\, B_r$ are the left and right boundaries of $R$.
\begin{proposition} \label{prop:Delta:bd} Fix $k\ge 1$ and $0<\delta<\frac13$ and let $\ell^{\frac23}(\log \ell)^{\frac13+\delta}\le w\le \alpha_0\ell$ with $\alpha_0$ as in Lemma \ref{lem:s:t}. There exists a constant $C=C(\delta,k)$ such that
\bes
\left \| \Delta(R) \right\|_k \le  C\frac{w^2}{\ell}
\ees
for large enough $\ell$.
\end{proposition}
\begin{proof}
We first bound $\Delta(R)$ as follows.
\begin{align}\label{eq:Del:bd}\begin{split}
\Delta(R) &\le \max_{\vx\in B_l, \, \vy\in B_r} \big[ L(\vx,\vy) -\E L(\vx,\vy)\big]  + \left[\max_{\vx\in B_l, \, \vy\in B_r} \E L(\vx,\vy) -\min_{\vx\in B_l, \, \vy\in B_r} \E L(\vx,\vy) \right] \\
&\qquad  - \min_{\vx\in B_l, \, \vy\in B_r} \big[ L^R(\vx,\vy) -\E L(\vx,\vy)\big] \\
&\le \max_{\vx\in B_l, \, \vy\in B_r} \big[ L(\vx,\vy) -\E L(\vx,\vy)\big]  + \left[\max_{\vx\in B_l, \, \vy\in B_r} \E L(\vx,\vy) -\min_{\vx\in B_l, \, \vy\in B_r} \E L(\vx,\vy) \right] \\
&\qquad  - \min_{\vx\in B_l, \, \vy\in B_r} \big[ \tilde{L}^R(\vx,\vy) -\E L(\vx,\vy)\big]
\end{split}
\end{align}
where $\tilde{L}^R(\vx,\vy)$ are restricted paths from $\vx$ to $\vy$ passing through the sub-rectangles constructed as in Proposition \ref{prop:ER:lbd}.
By Lemma \ref{lem:s:t},
\bes
 \left[\max_{\vx\in B_l, \, \vy\in B_r} \E L(\vx,\vy) -\min_{\vx\in B_l, \, \vy\in B_r} \E L(\vx,\vy) \right] \le C\frac{ w^2}{\ell}.
\ees
Let us next look at the first term on the right hand side of \eqref{eq:Del:bd}. We divide both $B_l$ and $B_r$
into intervals of length $\ell^{-1}$. Let $\va_0,\,\va_1,\cdots, \va_{[w\ell]}$ denote the partition points
on the left boundary and let $\vb_0,\,\vb_1,\cdots, \vb_{[w\ell]}$ denote the partition points on the right boundary. We claim
that
\bes
\Big\|\max_{\vx\in B_l, \, \vy\in B_r} \big[ L(\vx,\vy) -\E L(\vx,\vy)\big]\Big\|_k   \le \Big\|\max_{i,j} \big[ L(\va_i,\vb_j) -\E L(\va_i,\vb_j)\big]\Big\|_k  +  C(k) \frac{w^2}{\ell}.
\ees
This is because with overwhelming probability the number of points in each rectangle
with vertices $\va_i,\va_{i+1},\vb_{j},\vb_{j+1}$ is at most $(\log \ell)^2$.
 An application of the tail bounds in Lemma \ref{lem:s:t}
along with a union bound gives
\begin{align}\label{eq:tail:ab}
\pr\left(\left | L(\va_i,\vb_j) -\E L(\va_i, \vb_j)\right | > \ell^{\frac13} (\log \ell)^{\frac23+2\delta}  \text{ for some } i,j \right) \le \exp\big(-c(\log \ell)^{1+3\delta}\big).
\end{align}
From this it follows that
\bes
\Big\|\max_{i,j} \big[ L(\va_i,\vb_j) -\E L(\va_i,\vb_j)\big]\Big\|_k  \le C(\delta,k) \cdot \ell^{\frac13}(\log \ell)^{\frac23+2\delta} \le C(\delta,k)\frac{w^2}{\ell}
\ees
which bounds the first term in \eqref{eq:Del:bd}.
Now we proceed to the last term in \eqref{eq:Del:bd}. The $k$th
norm is bounded by
\bes\begin{split}
& \le \Big\|\max_{\vx,\vy} \big \vert \tilde{L}^R(\vx,\vy) -\E \tilde{L}^R(\vx,\vy)\big\vert \Big\|_k +\max_{\vx,\vy} \big\vert \E \tilde{L}^R(\vx,\vy) -\E L(\vx,\vy)\big\vert \\
&= \Big\|\max_{i,j} \big \vert \tilde{L}^R(\va_i,\vb_j) -\E \tilde{L}^R(\va_i,\vb_j)\big\vert \Big\|_k +\max_{\vx,\vy} \big\vert \E \tilde{L}^R(\vx,\vy) -\E L(\vx,\vy)\big\vert  \\
&\le \Big\|\max_{i,j} \big \vert \tilde{L}^R(\va_i,\vb_j) -\E \tilde{L}^R(\va_i,\vb_j)\big\vert \Big\|_k + C(\delta,k)\frac{w^2}{\ell}
\le C(\delta,k)\frac{w^2}{\ell}.
\end{split}\ees
The second last inequality follows by an application of Proposition \ref{prop:ER:lbd}. Since the number
of rectangles in the construction of Proposition \ref{prop:ER:lbd} is $O(\log (w/\ell^{1/3}))$ and the behavior of the
path in each sub-rectangle is unrestricted, we can use a bound similar to \eqref{eq:tail:ab} to conclude
the last inequality.
\end{proof}
We next give a bound on the $k$th centered moments of
\bes
L^R := \max_{\vx\in B_l,\, \vy\in B_r} L^R(\vx,\vy).
\ees
\begin{lemma}  \label{lem:xns:bd} Fix $k\ge 1$ and $0<\delta<\frac13$ and let $\ell^{\frac23}(\log \ell)^{\frac13+\delta}\le w\le \alpha_0\ell$ with $\alpha_0$ as in Lemma \ref{lem:s:t}. There exists a constant $C=C(\delta,k)$ such that
\bes
\big\| L^R -\E L^R\big\|_k \le C \frac{w^2}{\ell}.
\ees
for large enough $\ell$.
\end{lemma}
\begin{proof}
We write
\be \begin{split} \label{eq:lr-elr}
L^R -\E L^R & =\big[ L^R - L^R\big( (0,w/2), (\ell,w/2)\big) \big]  +  \big[ L^R\big( (0,w/2), (\ell,w/2)\big) - \E  L^R\big( (0,w/2), (\ell,w/2)\big) \big] \\
&\hspace{3cm}  +\big[ \E  L^R\big( (0,w/2), (\ell,w/2)\big)-\E L^R \big]
\end{split} \ee
An application of Proposition \ref{prop:Delta:bd} gives us that
\bes \begin{split}
 \big\| L^R - L^R\big( (0,w/2), (\ell,w/2)\big) \big\|_k &\le \| \Delta (R) \|_k \le C(\delta,k)\frac{w^2}{\ell} \\
 \big |\E L^R - \E L^R\big( (0,w/2), (\ell,w/2)\big) \big|  & \le \| \Delta (R) \|_k \le C(\delta,k)\frac{w^2}{\ell}
\end{split}\ees
As for the second term on the right hand side of \eqref{eq:lr-elr} we observe by Lemma \ref{lem:tr_fl}
 that with overwhelming probability
the maximal path from $(0,w/2)$ to $(\ell,w/2)$ lies within the rectangle $R$. Hence
\bes
\big\| L^R\big( (0,w/2), (\ell,w/2)\big) - \E  L^R\big( (0,w/2), (\ell,w/2)\big) \big \|_k \le C(\delta,k)\frac{w^2}{\ell}.
\ees
since the unrestricted LIP from $(0,w/2)$ to $(\ell,w/2)$ has moments of order $\ell^{1/3}$ by \eqref{eq:mom:bd}.
This completes the proof of the lemma.\end{proof}

\section{Proof of Theorem~\ref{thm:exp}} \label{sec:exp}
Fix $0<\delta<\frac16$. Let us consider the lower bound first. Choose $n$ so that $w=n^{\frac23}(\log n)^{\frac13+\delta}$.
 We divide $R$ into $m=\lfloor \ell/n \rfloor$ blocks of length $n$
and the final block of length smaller than $n$. For each of the blocks $R_i$,
let $L^R_i$ denote the length of the restricted LIP from the midpoint of the left boundary to the midpoint
of the right boundary. Note that by Lemma \ref{lem:tr_fl} the restricted LIP between the midpoints
coincides with the unrestricted LIP with overwhelming probability. Thus
\bes
\E L^R_i = \sqrt 2 n - c_1n^{1/3}\big(1+o(1)\big), \; 1\le i \le m
\ees
and similarly for the final block. One gets immediately
\bes
\E L^R \ge \sum_{i=1}^{m+1} \E L^R_i \ge \sqrt 2 \ell - C \cdot m n^{1/3} \ge \sqrt 2 \ell - C(\delta) \cdot (\log w)^{\frac13+\delta}\cdot\frac{\ell}{w}.\ees

Let us now turn to the upper bound in \eqref{eq:e:lr}. For this we choose $n$ such that $n^{\frac23}(\log n)^{\frac13+\delta} \le w \le 2n^{\frac23}(\log n)^{\frac13+\delta}$ such that $\frac{\ell}{n(\log n)^3}$ is a positive
integer. We provide some details on why we can choose such an $n$. It is easy to see that for any such $n$ we have that $n(\log n)^3 \le c(\delta) w^{\frac32}(\log w)^{\frac52-\frac{3\delta}{2}} \ll l$. For the exact divisibility first choose $n_1$ so that $w= 2n_1^{\frac23}(\log n_1)^{\frac13+\delta}$. When dividing $\ell$ by $n_1$ the remainder is at most $n_1(\log n_1)^3$. It then follows that we can choose $n_2$ appropriately so that $n_1\le n_2\le 2n_1$ and $n_2(\log n_2)^3$ exactly divides $\ell$. It is clear that $w\le 2n_2^{\frac23}(\log n_2)^{\frac13+\delta}$ and $ w\ge 2 (n_2/2)^{\frac23}\big(\log(n_2/2)\big)^{\frac13+\delta} \ge n_2^{\frac23}(\log n_2)^{\frac13+\delta}$ when $w$ is large enough.

Let us therefore now assume that we have an $n$ such that $n^{\frac23}(\log n)^{\frac13+\delta} \le w \le 2n^{\frac23}(\log n)^{\frac13+\delta}$ and $n'=n(\log n)^3$ exactly divides $\ell$. We
split $R$ into $m'=\ell/n'$ blocks $R_i$ of length
 $n'$.  For each $1\le i\le m'$ let $X_i$ denote
the length of the restricted LIP in $R_i$ and let $L^R_i$ as before denote the length of the restricted (midpoint to midpoint)
LIP. Divide $R_i$ further into three sub-rectangles $S_i^{(1)}, \, S_i^{(2)}, \, S_i^{(3)}$
where the $S$-rectangles on either side have length $n$ and $S_i^{(2)}$ has the
remaining length. Denote by $Y_i$ the length of restricted
LIP in $S_i^{(2)}$.
\begin{figure}[htbp]
   \centering
   \includegraphics[width=400pt]{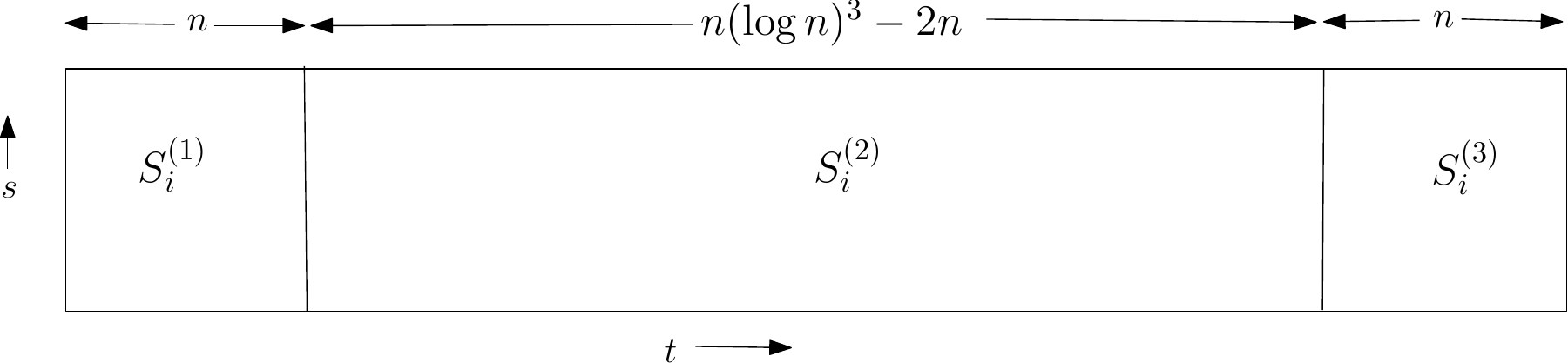}
      \caption{Division of block $R_i$}
\end{figure}
It is clear that
\bes
\begin{split}
X_i & \le \max_{\vx_1, \vy_1} L^R(\vx_1,\vy_1) + Y_i + \max_{\vx_3, \vy_3} L^R(\vx_3,\vy_3) \\
L^R_i & \ge \min_{\vx_1, \vy_1} L^R(\vx_1,\vy_1) + Y_i + \min_{\vx_3,\vy_3} L^R(\vx_3,\vy_3)
\end{split}
\ees
where $\vx_1,\vy_1$ are generic points on the left and right boundaries of $S_i^{(1)}$, and
 $\vx_3,\vy_3$ are generic points on the left and right boundaries of $S_i^{(3)}$ . Therefore
\bes
\E \vert X_i - L^R_i\vert \le 2 \|\Delta\big(S_i^{(1)}\big)\|_1 \le C(\delta)\frac{w^2}{n}
\ees
by Proposition \ref{prop:Delta:bd}. This would then give
\bes
\E L^R \le  \sum_{i=1}^{m'} \E X_i \le \sum_{i=1}^{m'} \big(\E L^R_i + C(\delta)\frac{w^2}{n}\big).
\ees
We can bound $\E L^R_i$ by the expected value of the length of the unrestricted (midpoint to midpoint) LIP to get an
upper bound of $ \sqrt 2 n'- c_1(n')^{1/3}\big(1+o(1)\big) $ and thus
\bes
\E L^R \le \sqrt 2\ell - c m' (n')^{\frac13} + C(\delta)\cdot m'\frac{w^2}{n} \le \sqrt 2 \ell -  \frac{C(\delta)}{(\log n)^{\frac53-\delta}}\cdot\frac{\ell}{w}.
\ees
This gives the required upper bound because $\log n$ is of the same order as $\log w$.
\qed

\section{Meeting of Paths} \label{sec:merge}
Given $\delta, \delta'>0$ and $n$ consider a diagonal rectangle $R$ of width $\wb=\wb(n)$ and length $\ellb=\ellb(n)$
with
\begin{equation}  \label{eq:s:bar}
  \wb = n^{\frac23}(\log n)^{\frac{1}{3} + \delta}.
\end{equation}
and
\begin{equation}\label{eq:tau:n1}
 \ellb= n\big[\sqrt{\log n}\big]+2\tau \quad \text{where}\quad   \tau:=\frac{n}{(\log n)^{\delta'}}.
\end{equation}
Rectangles of this type will form the {\it basic} blocks in Section~\ref{sec:var:lbd}.
In this section we give a lower bound on the probability that LIPs between arbitrary starting
and ending points within the blocks meet. This would help us get the necessary independence
structure needed for getting a lower bound on the variance done in Section~\ref{sec:var:lbd}.

Denote by $\va$ a generic point on the left boundary of $R$ and by
$\vb$ a generic point on the right boundary of $R$. Define the event
\begin{equation}\label{eq:Om}
  \Omega  :=\{\text{There exists $\vp\in R$ such that for all $\va,\, \vb$ there is a path in $\cL^R(\va,\vb)$ passing through $\vp$}\}.
\end{equation}
By geometric considerations one observes that
\begin{equation*}
  \Omega = \{\text{A path in $\cL^R\big((0,0),(\ellb,0)\big)$ intersects a path in
  $\cL^R\big((0,\wb),(\ellb, \wb)\big)$}\}.
\end{equation*}
In fact any point in the intersection of a path $\tilde p$ in $\cL^R\big((0,0),(\ellb,0)\big)$ and a path $\tilde q$ in $\cL^R\big((0,\wb),(\ellb, \wb)\big)$ would be such a point $\mathbf p$. This is because from any path $\tilde r$ in $\cL^R(\va,\vb)$ which intersects either of these paths, one could construct another path in $\cL^R(\va,\vb)$ passing through the point $\mathbf p$. Note first that the path $\tilde r$ would either intersect $\tilde p$ or $\tilde q$ at least two points. Without loss of generality assume that it intersects $\tilde p$ at points $p_1$ (the first intersecting point) and $p_2$ (the last intersecting point). Construct another path following $\tilde r$ till $p_1$, then following $\tilde p$ from $p_1$ to $p_2$, and finally following $\tilde r$ from $p_2$ to $\vb$. One can argue that the length of this new path between $p_1$ and $p_2$ is the same as the length of $\tilde r$ between $p_1$ and $p_2$. Having a greater length would be a contradiction to the maximality of $\tilde r$ while having a lower length would be a contradiction to the maximality of $\tilde p$.\begin{center}
\begin{figure}[htbp]
   \centering
   \includegraphics[width=400pt]{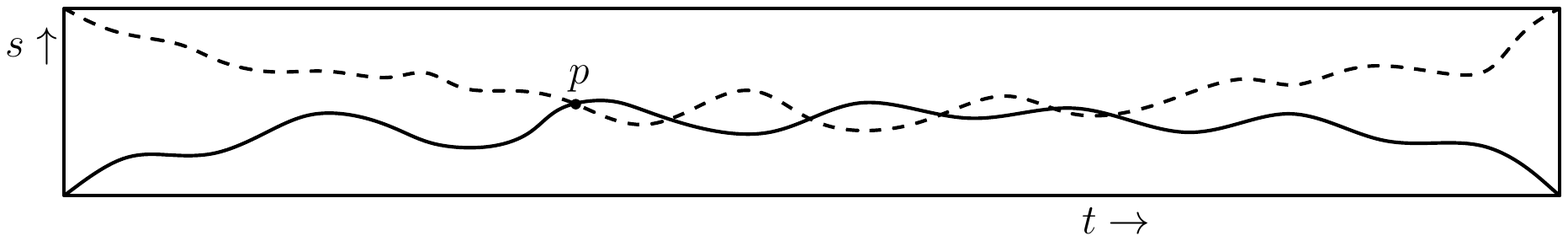}
   \caption{Picture depicting $\Omega \text{ and } \vp $}
   \label{fig:2paths}
\end{figure}
\end{center}
The following is the main result of the section.
\begin{theorem}\label{thm:Om} Fix $\delta,\delta'$ satisfying
\begin{equation}\label{eq:del:eps}
  \quad 0<\delta <\frac{1}{12}\quad  \text{and}\quad
  0<\delta'<\frac{1}{6}-2\delta.
\end{equation}
Let $R$ be a diagonal rectangle with length $\ellb=\ellb(n)$ and width $\wb=\wb(n)$ satisfying
\eqref{eq:s:bar} and \eqref{eq:tau:n1}.
Then for any $0<\epsilon<1$
we have
\begin{equation} \label{eq:prob:Om}
  \pr(\Omega)\ge n^{-\epsilon}
\end{equation}
for all $n$ large enough.
\end{theorem}
Fix $0<\epsilon<1$ and $\delta, \delta'$ satisfying \eqref{eq:del:eps}. Define the following events.
\begin{equation} \label{eq:a12}
\begin{split}
  A_1&:=\left\{L^R((0,0),(\ellb,0))\ge E_{\ellb} + 10 \eps^{\frac23}n^{\frac13}( \log
  n)^{\frac56}\right\},\\
  A_1'&:=\left\{L^R((0,\wb),(\ellb,\wb))\ge E_{\ellb} + \eps^{\frac23}n^{\frac13}( \log
  n)^{\frac56}\right\},\\
  A_2&:=\left\{A_1'\text{ does not occur disjointly from $A_1$}\right\}= \left(A_1\Box A_1'\right)^c.
\end{split}
\end{equation}
Our approach to lower bound $\pr(\Omega)$ proceeds
by showing that $A_1$ has probability $n^{-c\eps}$ and $A_2$ is
likely given $A_1$. Under the event $A_1$, the LIP between the
point $(0,0)$ and $(\ellb,0)$ is very long and the only potential way for
$(0,\wb)$ to be connected to $(\ellb,\wb)$ via a path with
similarly large length is by intersecting the first path. To make
this more precise, we need also control the loss incurred in
connecting a corner point of $R$ to a point near the central line of
$R$ (with $s$ coordinate roughly $\wb/2$).
To this end, define the events
\begin{equation} \label{eq:a34}
  \begin{split}
    A_3 & :=\left\{\forall\, 0\le s\le \wb,\;\,L^R\big((0,0),(\tau, s)\big)\le E_\tau + n^{\frac13}(\log
    n)^{\frac23}\right\},\\
    A_{4,s'}& :=\left\{L^R\big((0,0),(\tau, s')\big)\ge E_\tau -2\mathcal{C} n^{\frac13}(
    \log n)^{\frac23+2\delta +\delta'}\right\},\qquad 0\le s'\le \wb
  \end{split}
\end{equation}
and the symmetric events
\be \label{eq:a'34} \begin{split}
 A_3'  &:=\left\{\forall\, 0\le s\le \wb,\;\,L^R\big((\ellb-\tau,s),(\ellb,0)\big)\le E_\tau + n^{\frac13}(\log
    n)^{\frac23}\right\} \\
A_{4,s'}'  &:=\left\{L^R\big((\ellb-\tau,s'),(\ellb, 0)\big)\ge E_\tau -2\mathcal{C} n^{\frac13}(
    \log n)^{\frac23 + 2\delta +\delta'}\right\},\qquad 0\le s'\le \wb.
 \end{split}\ee
where $\mathcal{C}=\mathcal{C}(\delta)$ is the constant appearing in Proposition \ref{prop:ER:lbd}. The following lemmas gather bounds which will be combined to prove Theorem \ref{thm:Om}.

\begin{lemma}
 There exists a constant $c>0$ independent of $\epsilon$ and $n_0= n_0(\delta, \delta', \epsilon)$ such that for
  all $n\ge n_0$ we have
  \begin{equation}\label{eq:a2}
    \pr(A_2^c | A_1) \le n^{-c\eps}.
   \end{equation}
\end{lemma}
\begin{proof}
 The estimate for $\pr(A_2^c | A_1)$ follows by an application of Lemma \ref{lem:bk} and comparison with LIP which are not constrained to lie in $R$.
  \begin{equation*}
    \pr(A_2^c | A_1)\le \pr(A_1')\le \pr\big(L((0,\wb),(\ellb,\wb))\ge E_{\ellb} + \eps^{\frac23}n^{\frac13}( \log
  n)^{\frac56}\big)\le Cn^{-c\eps},
  \end{equation*}
the last inequality following from \eqref{eq:sig:asy} and \eqref{eq:up:tail}.
\end{proof}

\begin{lemma}
 There exist constants $c, C>0$ and $n_0= n_0(\delta, \delta')$ such that for
  all $n\ge n_0$ we have
  \begin{equation}
 \pr(A_3^c)\le C\exp\big(-c(\log n)^{1 + \frac{\delta'}{2}}\big), \label{eq:a3}
   \end{equation}
A similar bound as \eqref{eq:a3} holds for $\pr\big((A'_3)^c\big)$.
\end{lemma}
\begin{proof}
As in Proposition \ref{prop:Delta:bd} we might as well discretize the possible
  values of $s$ to only about $\tau\wb$ possible values at distance $\tau^{-1}$ apart;
 see the argument leading up to \eqref{eq:tail:ab}. For each value of $s$, we compare with the LIP which are not constrained to lie in $R$ to obtain
   \begin{equation*}\begin{split}
    \pr\big[L^R\big((0,0),(\tau, s)\big)> E_\tau + n^{\frac13}(\log
    n)^{\frac23}\big]
    &  \le \pr\big[L\big((0,0),(\tau, s)\big)> E_\tau + n^{\frac13}(\log
    n)^{\frac23}\big] \\
& = \pr\big[L\big((0,0),(\tau, s)\big)> E_\tau + \tau^{\frac13}(\log n)^{\frac{\delta'}{3}}(\log
    n)^{\frac23}\big] \\
& \le C\exp\big(-c(\log n)^{1 + \frac{\delta'}{2}}\big),
  \end{split}\end{equation*}
the last inequality following from Lemma~\ref{lem:s:t}. The lemma follows from a union bound.
\end{proof}

\begin{lemma} For any $D>0$ and
  $0\le s'\le \wb$ we have
\begin{equation}
\pr(A_{4,s'}^c)= o(n^{-D}) . \label{eq:a4}
\end{equation}
A similar bound as \eqref{eq:a4} holds for $\pr\big((A'_{4,s'})^c\big)$.
\end{lemma}
\begin{proof} The argument is based on the proof of Proposition \ref{prop:ER:lbd}
 to which we shall refer. Fix $D>0$. It is enough to prove
\bes \pr \left\{L^R\big(\va,(\tau/2, \wb/2)\big)\le E_{\frac{\tau}{2}} -\mathcal{C} n^{\frac13}(
    \log n)^{\frac23+2\delta +\delta'}\right\} \le \frac{1}{n^{D}}
 \ees
where $\va$ is a point on the left boundary. This is because we have a similar bound for
$L^R\big(\big(\frac{\tau}{2}, \frac{\wb}{2}\big), \vb\big)$ where $\vb=(\tau,s')$. Let us assume without
loss of generality that $\va=0$ as in the proof of Proposition \ref{prop:ER:lbd}.
  We have shown there that
\bes
\sum_{k=0}^{k_0-1}E\big((t_k,s_k), (t_{k+1},s_{k+1})\big) + E\big((t_{k_0},s_{k_0}), (\tau/2, \wb/2)\big) \ge E_{\frac{\tau}{2}} -\frac{\mathcal{C}}{2}n^{\frac13} (\log n)^{\frac{2}{3} +2\delta + \delta'}.
\ees
By superadditivity it follows that
\bes \begin{split}
L^R\big((0,0),(\tau/2, \wb/2)\big) &\ge \sum_{k=0}^{k_0-1}L^R\big((t_k,s_k), (t_{k+1},s_{k+1})\big) + L^R\big((t_{k_0},s_{k_0}), (\tau/2, \wb/2)\big) \\
& =   \sum_{k=0}^{k_0-1}L\big((t_k,s_k), (t_{k+1},s_{k+1})\big) + L\big((t_{k_0},s_{k_0}), (\tau/2, \wb/2)\big),
\end{split}
\ees
the equality holding outside a set of probability less than $C\exp\big(-c(\log n)^{1+3\delta}\big)$ by
Lemma \ref{lem:tr_fl}, and hence negligible compared to the bound we are trying to prove. Thus it is enough to show that
\begin{align}\begin{split} \label{eq:l-e}
 & \sum_{k=0}^{k_0-1}\Big\vert L\big((t_k,s_k), (t_{k+1},s_{k+1})\big) -E\big((t_k,s_k), (t_{k+1},s_{k+1})\big)\Big\vert \\
 &\hspace{2cm}+ \Big\vert L\big((t_{k_0},s_{k_0}), (\tau/2, \wb/2)\big) - E\big((t_{k_0},s_{k_0}), (\tau/2, \wb/2)\big)\Big \vert\le \frac{\mathcal{C}}{2}n^{\frac13} (\log n)^{\frac23 +2\delta + \delta'}
\end{split}
\end{align}
with sufficiently large probability. This follows from the tail bounds in Lemma \ref{lem:s:t}.
Indeed for sufficiently large $B$ one gets
\bes
\pr\Big(\Big\vert L\big((t_k,s_k), (t_{k+1},s_{k+1})\big) -E\big((t_k,s_k), (t_{k+1},s_{k+1})\big)\Big\vert \ge (t_{k+1}-t_k)^{\frac13} (B\log n)^{\frac23}\Big) \le \frac{1}{n^{ D}},
\ees
the above holding for each $k$. A similar bound holds for the last difference in \eqref{eq:l-e}. Since $k_0= O(\log n)$,
\bes \begin{split}
&\sum_{k=0}^{k_0-1}\Big\vert L\big((t_k,s_k), (t_{k+1},s_{k+1})\big) -E\big((t_k,s_k), (t_{k+1},s_{k+1})\big)\Big\vert \\
&\hspace{2cm}+ \Big\vert L\big((t_{k_0},s_{k_0}), (\tau/2, \wb/2)\big) - E\big((t_{k_0},s_{k_0}), (\tau/2, \wb/2)\big)\Big \vert \\
 &= O\Big(\sum_{k=0}^{k_0-1} (t_{k+1}-t_k)^{1/3} (B\log n)^{\frac23} + \big(\frac{\tau}{2}-t_{k_0}\big)^{\frac13} (B\log n)^{\frac23} \Big) \\
 &= O\Big(n^{\frac13} (\log n)^{\frac23 -\frac{\delta'}{3}} \Big)
\end{split}
\ees
outside of a set of probability $\frac{1}{n^{D}}$. This gives the bound \eqref{eq:a4} as required. \end{proof}

\begin{lemma} There exist constants $C>0$ independent of $\epsilon$ such that the following holds.  There exists $n_0= n_0(\delta, \delta', \epsilon)$ such that for
  all $n\ge n_0$ we have
\begin{equation}
\pr(A_1)\ge n^{-C\eps}. \label{eq:a1}
\end{equation}
\end{lemma}
\begin{proof} By superadditivity it follows that
  \begin{equation*}
    L^R\big((0,0),(\ellb,0)\big) \ge L^R\big((0,0),(\tau, \wb/2)\big)
    + L^R\big((\tau, w/2),
    (\ellb-\tau, \wb/2)\big) + L^R\big((\ellb-\tau,
    \wb/2), (\ell,0)\big).
  \end{equation*}
Call $ A_4=A_{4,\frac{\wb}{2}}$ and $ A_4'=A_{4,\frac{\wb}{2}}'$.  Note that on the event $A_{4}\cap A_{4}' $,
  \begin{equation*}
    L^R\big((0,0),(\ellb,0)\big) \ge 2E_\tau - 4\mathcal{C} n^{\frac13}(
    \log n)^{\frac23+2\delta+\delta'} + L^R\big((\tau, \wb/2),
    (\ellb-\tau, \wb/2)\big).
  \end{equation*}
  It is then easy to deduce that
  \begin{equation*}\begin{split}
    &\pr(A_1)  \ge \pr\Big[L^R\big((\tau, \wb/2),
    (\ellb-\tau,\wb/2)\big)\ge E_{\ellb} + 10\eps^{\frac23}n^{\frac13}( \log
  n)^{\frac56} - 2E_\tau + 4\mathcal{C}n^{\frac13}(
    \log n)^{\frac23+2\delta +\delta'}\Big]  \\
    &\hspace{4cm}- 2\pr\big(A_{4}^c\big).
  \end{split}\end{equation*}
  The term $\pr(A^c_{4})$
  decays very fast in $n$ as we showed in \eqref{eq:a4} and we therefore focus on the first term
  on the right hand side. We note
  that by \eqref{eq:E:asy} we have
  \begin{equation*}
    |E_{\ellb} - 2E_\tau - E_{\ellb-2\tau}|\le Cn^{1/3}(\log n)^{1/6}.
  \end{equation*}
  Thus for large enough $n$,
  \begin{equation}\label{eq:tau:n_1}\begin{split}
   & \pr\left[L^R\big((\tau, \wb/2),
    (\ellb-\tau, \wb/2)\big)\ge E_{\ellb} + 10\eps^{\frac23}n^{\frac13}( \log
  n)^{\frac56} - 2E_\tau + 4\mathcal{C}n^{\frac13}(
    \log n)^{\frac23+2\delta +\delta'}\right] \\
    &\qquad\ge \pr\left[L^R\big((\tau, \wb/2),
    (\ell^{(b)}-\tau, w^{(b)}/2)\big)\ge E_{\ellb-2\tau} + 13 \eps^{\frac23}n^{\frac13}( \log
  n)^{\frac56}\right]
 \end{split} \end{equation}
 and we focus our attention on the right hand side. We split the $t$-interval
 $(\tau,\ellb-\tau)$ into $\big[\sqrt{\log n}\big]$ intervals of length $n$ each and observe
 \bes
 L^R\big((\tau, \wb/2), (\ellb-\tau, \wb/2)\big) \ge \sum_{i=0}^{[\sqrt{\log n}]-1}L^R\big((\tau + in, \wb/2), (\tau +(i+1)n, \wb/2)\big).
 \ees
 \begin{center}
\begin{figure}[htbp]
   \centering
   \includegraphics[width=400pt]{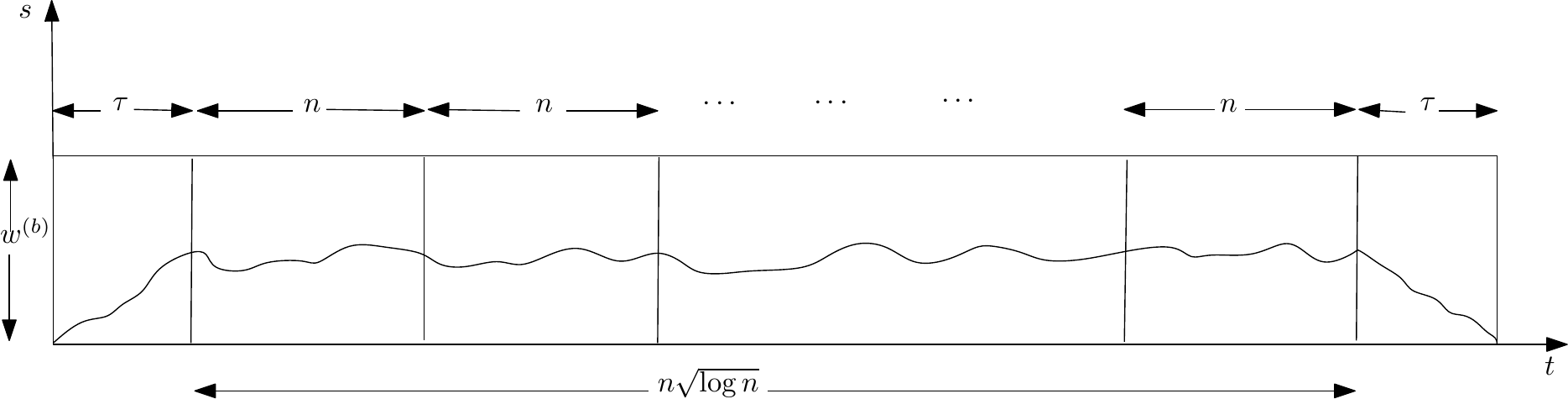}
   \caption{Construction of a long path in $A_1$}
\end{figure}
\end{center}
 Using independence between different blocks,
 the right hand side of \eqref{eq:tau:n_1} is bounded below by
 \bes
 \left(\pr \left[L^R\big((\tau , \wb/2), (\tau +n, \wb/2)\big) \ge \frac{E_{\ellb-2\tau}}{\big[\sqrt{\log n}\big]} + 14\eps^{\frac23} n^{\frac13} (\log n)^{\frac13}\right]\right)^{[\sqrt{\log n}]}.
 \ees
It is an easy computation to check that $E_{\ellb-2\tau} -\big[\sqrt{\log n}\big] E_n= O\big(n^{\frac13}\sqrt{\log n}\big)$
and thus for large $n$ we have a lower bound of
 \bes
 \left(\pr \left[L^R\big((\tau , \wb/2), (\tau +n, \wb/2)\big) \ge E_n + 15\eps^{\frac23} n^{\frac13} (\log n)^{\frac13}\right]\right)^{[\sqrt{\log n}]}.
 \ees
 By the transversal fluctuations Lemma~\ref{lem:tr_fl}, we may lift the restriction
  of staying in $R$ from the last probability to get a lower bound of
  \begin{equation*} \begin{split}
    &\Big\lbrace\pr \left[L\big((\tau , \wb/2), (\tau +n, \wb/2)\big) \ge E_n + 15\eps^{\frac23} n^{\frac13} (\log n)^{\frac13}\right] -\exp\big(-c(\log n)^{1+3\delta}\big)\Big\rbrace^{[\sqrt{\log n}]} \\
&\ge \left\lbrace\exp\left(-C\eps (\log n)^{\frac12}\right)-\exp\big(-c(\log n)^{1+3\delta}\big)\right\rbrace^{[\sqrt{\log n}]} \ge \frac{1}{n^{C\eps}}
  \end{split}\end{equation*}
 for large $n$, where we used \eqref{eq:up:tail} in the last line.\end{proof}
We can now combine all the lemmas above to prove Theorem~\ref{thm:Om}.
\begin{proof}[Proof of Theorem~\ref{thm:Om}] Recall the events in \eqref{eq:a12}, \eqref{eq:a34} and \eqref{eq:a'34}. We have that
\begin{equation} \label{eq:contain}
  \Omega\supseteq B_0\cap A_1\cap A_2\cap \tilde{A}_3\cap
  \tilde{A}_4.
\end{equation}
Here $\tilde{A}_3$ is the event $A_3\cap A_{3}'$ and the event $\tilde{A}_4 =\bigcap_{s'}\big(A_{4,s'}' \cap A_{4,s'}'\big)$, the intersection being over a discretization of $0\le s'\le \wb$ spaced distance $[\ell^{(b)}]^{-1}$ apart. The event $B_0$ is that each of the
strips of width $[\ell^{(b)}]^{-1}$ and length $\ell^{(b)}$ has at most $[\log \ell^{(b)}]^2$ points. One can bound the probability of this event just as in Lemma \ref{lem:grid}.
 Using \eqref{eq:a4} one obtains $\pr\big(\tilde{A}_4^c\big) \le \frac{1}{n^{D}}$
for some large $D$ since there are only $w^{(b)} \ell^{(b)}$ points in the discretization.

The reason for the
containment  \eqref{eq:contain} is that under $A_1$ there is a long (point to point, restricted) LIP from $(0,0)$ to
$(\ellb,0)$. Note here that since $B_0$ occurs, one can assume that this path passes through one of the discretization points $s'$, since the loss in length is comparitively small. As $\tilde{A}_3$ occurs, the path does not have a long subpath
from $(0,0)$ to $(\tau, s)$ for any $s$ or from $(\ellb-\tau, s)$ to
$(\ellb,0)$ for any $s$. Then, as $A_2$ occurs, there is no similarly
long LIP from $(0,\wb)$ to $(\ellb,\wb)$ which is disjoint from
the first LIP. However, as $\tilde{A_4}$ occurs, it is possible to
start from $(0,\wb)$, take an LIP to $(\tau, s')$ for some
suitable $s'$, continue along the portion of the LIP from $(0,0)$ to
$(\ellb,0)$ until some $(\ellb-\tau, s')$ for another suitable $s'$, and then take an
LIP to $(\ellb,\wb)$ to altogether obtain a very long LIP from
$(0,w)$ to $(\ellb,\wb)$. Thus $\Omega$ must occur.

 Using the probability estimates \eqref{eq:a3}, \eqref{eq:a4} and a similar bound as Lemma \ref{lem:grid} to bound $\pr(B_0)$ we get
\[\pr(B_0\cap\tilde A_3 \cap \tilde A_4) \ge 1-n^{-D}.\]
for all large enough $D$. On the other hand \eqref{eq:a2} and \eqref{eq:a1} give
\[\pr(A_2\cap A_1)\ge \pr(A_1)\cdot \left(1-n^{-c\epsilon}\right).\]
for all large $n$. Thus
\begin{equation*} \begin{split}
  \pr(\Omega) &= \pr(A_1\cap A_2)  + \pr(B_0\cap\tilde A_3 \cap \tilde A_4) -\pr\left((A_1\cap A_2)\cup (B_0\cap\tilde A_3 \cap \tilde A_4)\right) \\
  &\ge  \pr(A_1\cap A_2) +\pr(B_0\cap\tilde A_3 \cap \tilde A_4) -1\\
  &\ge n^{-c\epsilon}
\end{split}\end{equation*}
for some $c>0$ and all large $n$, as we wanted to prove.
\end{proof}

\section{Proof of Theorem~\ref{thm:var:lbd}} \label{sec:var:lbd}

 Fix $\delta, \delta'$ satisfying the conditions of Theorem~\ref{thm:Om} and
find $n$ so that $w=w^{(b)}= n^{\frac23}(\log n)^{\frac{1}{3}+\delta}$. Now
split the rectangle $R$ into basic blocks of length $\ell^{(b)}= n\big[\sqrt{\log n}\big] +2\tau$
where $\tau = \frac{n}{(\log n)^{\delta'}}$ as defined in \eqref{eq:tau:n1}, with the last block of (possibly) smaller length. We showed in
Section \ref{sec:merge} that in each block with probability at least $n^{-c\eps}$ there exists
a point $\mathbf{p}$ ({\it regeneration point}) such that the following holds: For all points
 $\mathbf{a}, \mathbf{b}$
on the left and right boundaries of the block one can find a path in $\cL^R(\mathbf{a},\mathbf{b})$
 connecting
$\mathbf{a}$ and $\mathbf{b}$ and passing through $\mathbf{p}$. The regeneration points $\mathbf{p}$ give us the required independence
in order to establish the lower bound on the variance. It is clear that
\bes
L^R  = \text{Sum of restricted LIP between regeneration points}.
\ees
We condition on $\mathcal{G}$, the $\sigma$-algebra generated by the Poisson points
 in alternate blocks $1,3,\cdots$. Choose one regeneration point with respect to each odd block ({\it odd regeneration point}), if there are any present. List the odd regeneration points as $\mathbf{p}_1, \mathbf{p}_2, \cdots$ so that $\mathbf{p}_1$ is closest to the left boundary of $R$, $\mathbf{p}_2$ is the second closest and so on.
  By the conditional
variance formula
\bes \begin{split}
\var(L^R) & = \E\big(\var(L^R | \mathcal{G}) \big) +\var\big(\E(L^R |\mathcal{G})\big)   \\
&\ge \E\Big[ \sum_i \var( L^R(i)| \mathcal{G}) \Big]
\end{split}
\ees
where $L^R(i)$ is the length of LIP between the $\mathbf{p}_i$ and $\mathbf{p}_{i+1}$. This is further
greater than
\bes
\E\Big[ \sum_{i \in \mathcal{I}} \var( L^R(i)| \mathcal{G}) \Big]
\ees
where $\mathcal{I}$ is the collection of all odd indices $i$ such that the $\mathbf{p}_i$ and $\mathbf{p}_{i+1}$
are in consecutive odd blocks. Since the probability of having a regeneration
point in a block is at least $n^{-\eps}$, we get that
\bes
\var(L^R) \ge \frac{\ell n^{-2\eps}}{5\ellb} \E\big[\var( L^R(i)| \mathcal{G})\big],
\ees
where $i$ is any odd index such that $p_i$ and $p_{i+1}$ are in consecutive odd blocks. The rest of the argument finds a lower bound for $\var( L^R(i)| \mathcal{G})$.

It is clear that
\bes
L^R(i) \le L^R_1(i)+L^R_2(i)+L^R_3(i)
\ees
where $L^R_1(i)$ is the restricted LIP from $\mathbf{p}_i $ to the right side of the block
containing $\vp_i$, $L^R_2(i)$ is the restricted LIP in the even block separating the odd blocks
and $L^R_3(i)$ is the restricted LIP between the left side of the block containing $\vp_{i+1}$
and $\vp_{i+1}$.
\begin{figure}[htbp]
   \centering
   \includegraphics[width=400pt]{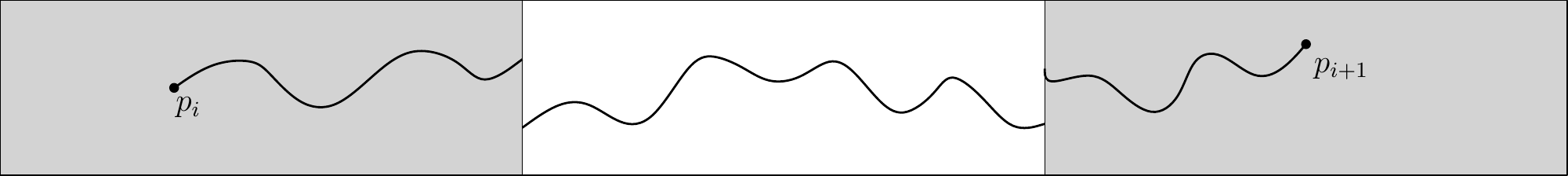}
   \caption{Picture illustrating $L^R_1, L^R_2, L^R_3$}
   \label{fig:L3}
\end{figure}

We have from the upper bound in \eqref{eq:e:lr} (with the rectangle
of length $\ell^{(b)}$ and width $w^{(b)}$) that
\begin{align}\label{eq:med:up}
\E L^R_2(i) \le \sqrt 2 \ellb-\frac{\ellb}{\wb(\log \wb)^{\frac53-\delta}}  \le E_{\ellb}+ 2 c_1 (\ellb)^{\frac13} = E_{\ellb}+ 3c_1 n^{\frac13}(\log n)^{\frac16}
\end{align}
from which it follows that
\begin{align} \label{eq:lri:up} \begin{split}
\E \big(L^R(i)\vert \mathcal{G}\big) &= L^R_1(i)+\E\big(L^R_2(i)\big) +L^R_3(i)\\
&\le L^R_1(i)+E_{\ellb}+ 3c_1 n^{\frac13}(\log n)^{\frac16} +L^R_3(i).
\end{split}
\end{align}
The argument used to prove \eqref{eq:a1} also shows that
\bes\label{eq:med:low}
\pr\left(L^R(\mathbf{c},\mathbf{d})\ge E_{\ellb}+\eps^{\frac23}n^{\frac13} (\log n)^{\frac56} \right) \ge n^{-C\eps}
\ees
where $\mathbf{c}$ and $\mathbf{d}$ are arbitrary points on the opposite sides of the even rectangle.
By taking $\mathbf{c}$ (resp. $\mathbf{d}$) to be the points at which the LIP from $\vp_i$ (resp. $\vp_{i+1}$)
hit the left (resp. right) sides of the even block, one gets
\begin{align}\label{eq:lri:low}
\pr\left(L^R(i) \ge L^R_1(i)+ E_{\ellb}+\eps^{\frac23}n^{\frac13} (\log n)^{\frac56}+L^R_3(i) \,\big\vert\, \mathcal{G} \right) \ge n^{-C\eps}
\end{align}
Combining \eqref{eq:lri:up} and \eqref{eq:lri:low} we get
\bes\begin{split}
\var(L^R) &\ge \frac{\ell n^{-2\eps}}{5\ellb} \E\big[\var( L^R(i)| \mathcal{G})\big]  \\
& \ge \frac{\ell n^{-2\eps}}{10\ellb}\cdot  n^{-C\eps}\cdot \eps^{\frac43}n^{\frac23} (\log n)^{\frac53} \ge \left(\frac{\ell}{w^{\frac12}}\right)^{1-C\eps}.
\end{split}\ees
This completes the proof of the theorem.
\qed
\section{Proof of Theorem~\ref{thm:mom:ubd}} \label{sec:mom:ubd}

  Fix $0<\delta<\frac16$ and find $n$ so that $w=n^{\frac23}(\log n)^{\frac13+\delta}$.
The proof is based on a recursive splitting of the rectangle $R$. We start with the rectangle $R$ of length $\ell_0=\ell$
and divide $R$ into three sub-rectangles, the second of which has length $n$, the first and last of length $\ell_1$
where $\ell_0=2\ell_1+n$. Next split both of the rectangles of size $\ell_1$ into three rectangles, the middle one of length
$n$ and the ones on the sides of the same length.  Perform this process recursively so that the lengths of the {\it side}
sub-rectangles are given by
\[\ell_i = 2\ell_{i+1}+n. \]
For a scale $\ell_i$, let $L^{R}(i)$ denote the length of the (restricted) maximal path in the left-most rectangle of length $\ell_i$, and let $L^{R}_1(i+1)$
and $L^{R}_2(i+1)$ denote the lengths of the (restricted) maximal paths of the sub-rectangles of size $\ell_{i+1}$ inside it.  Let $S$ denote
the middle rectangle of length $n$ and let $\vx$ (resp. $\vy$) denote an arbitrary point on the left (resp. right) boundary
of $S$.
It is clear that
\begin{align}\label{eq:l:ineq}
 L^R_1(i+1) +\min_{\vx,\vy} L^R(\vx,\vy) + L^R_2(i+1)   \le L^R(i) \le L^R_1(i+1) +\max_{\vx,\vy} L^R(\vx,\vy) + L^R_2(i+1)
\end{align}
Subtracting the means and moving terms we get from the above
\bes \begin{split}
&\min_{\vx,\vy}L^R(\vx,\vy) -\E \max_{\vx,\vy} L^R(\vx,\vy) \\
&\qquad  \le L^R(i)-\E L^R(i) -\big\{\big(L_1^R(i+1)- \E L_1^R(i+1)\big) +\big(L_2^R(i+1) - \E L_2^R(i+1)\big)\big\} \\
& \qquad \qquad  \le \max_{\vx,\vy}L^R(\vx,\vy) -\E \min_{\vx,\vy} L^R(\vx,\vy).
\end{split} \ees
An application of the triangle inequality then gives us
\begin{align}\label{eq:li:bd} \begin{split}
& \big\|L^R(i)-\E L^R(i) \big\|_{2k'} \\
& \quad \le  \big\|\big(L_1^R(i+1) - \E L_1^R(i+1)\big) +\big(L_2^R(i+1) - \E L_2^R(i+1)\big) \big\|_{2k'}  \\
&\qquad \quad+ 2 \big \| \max_{\vx,\vy}L^R(\vx,\vy) - \min_{\vx,\vy} L^R(\vx,\vy)\big\|_{2k'} + 2 \big\| \max_{\vx,\vy}L^R(\vx,\vy) - \E\max_{\vx,\vy} L^R(\vx,\vy)\big\|_{2k'} \\
&  \quad \le  \big\|\big(L_1^R(i+1) - \E L_1^R(i+1)\big) +\big(L_2^R(i+1) - \E L_2^R(i+1)\big) \big\|_{2k'}  + C(\delta, k') \cdot n^{\frac13} (\log n)^{\frac23+2\delta}
\end{split}
\end{align}
where we used Proposition \ref{prop:Delta:bd} and Lemma \ref{lem:xns:bd}
in the last step.

We
note now that $L_1^R(i+1)$ and $L_2^R(i+1)$ are independent and this will be useful
for us in decomposing the first term.
The above inequality is the main ingredient in the proof of the theorem. We shall
use it recursively until the rectangles in the last step are of size close to $n$.
The number of recursion steps will be
\[m = \left\lfloor \log_2\left(\frac{\ell}{n}\right) \right\rfloor - 1\]
so that
\[ \ell_m = \frac{\ell}{2^m}- n \left[\frac12+\frac{1}{2^2}+\cdots +\frac{1}{2^m}\right].\]
It is easy to see that $n \le \ell_m \le 4n$. We now claim

\begin{lemma} Fix $k\ge 1$. For $n$ large enough
\begin{align}\label{eq:claim}
\big\| L^R(i)- \E L^R(i)\big\|_{2k'} \le C(k) \cdot 2^{\frac{m-i}{2}} n^{\frac13} (\log n)^{\frac23+2\delta+k'}\quad  \text{ for } 1\le k'\le k,\, 0\le i\le m.
\end{align}
\end{lemma}
\begin{proof}
Let us first consider the case $k'=1$. In this case \eqref{eq:li:bd} becomes
\begin{align} \label{eq:li:bd:2}
\big\|L^R(i)-\E L^R(i) \big\|_{2} \le \sqrt 2 \big\|L^R(i+1)-\E L^R(i+1) \big\|_{2}+ C(\delta)\cdot n^{\frac13} (\log n)^{\frac23+2\delta}.
\end{align}
Proposition \ref{prop:Delta:bd} and Lemma \ref{lem:xns:bd} continue to be valid with length of rectangle $\ell_m$ and width $w=n^{\frac23}(\log n)^{\frac13+\delta}$ (we leave this as an exercise for the reader) and so
\bes
\big\|L^R(m)-\E L^R(m) \big\|_{2} \le C(\delta)\cdot n^{\frac13} (\log n)^{\frac23+2\delta}.
\ees
Using \eqref{eq:li:bd:2} recursively we get
\bes
\big\|L^R(i)-\E L^R(i) \big\|_{2} \le  C(\delta)\cdot n^{\frac13} (\log n)^{\frac23+2\delta} \left[1+ \sqrt 2+ \cdots + \left(\sqrt 2\right)^{m-i}\right]
\ees
which implies \eqref{eq:claim} for $k'=1$.

We next consider $k'\ge 2$.
We prove this by backwards induction on $i$. As above the claim is true for $i=m$ by
an application of Proposition \ref{prop:Delta:bd} and Lemma \ref{lem:xns:bd}.
So now suppose it is true for $i+1$. Call
\bes
\begin{split}
X(i+1)&= L_1^R(i+1) - \E L_1^R(i+1),\\
Y(i+1) &= L_2^R(i+1) - \E L_2^R(i+1).
\end{split}
\ees
By the induction hypothesis and the independence of $X(i+1)$ and $Y(i+1)$
 we have for large $n$
\bes\begin{split}
&\E \left[\left\{X(i+1)+Y(i+1)\right\}^{2k'}\right] \\
&= \E \left[X(i+1)^{2k'}\right] + \frac{2k' (2k'-1)}{2} \E \left[X(i+1)^{2k'-2}\right]\cdot \E \left[Y(i+1)^{2}\right]  + \cdots + \E \left[Y(i+1)^{2k'}\right]  \\
& \le \left\{ 2^{\frac{m-(i+1)}{2}} n^{\frac13} (\log n)^{\frac23+2\delta+k'} \right\}^{2k'} \\
&\qquad+ \frac{2k'(2k'-1)}{2}\left\{ 2^{\frac{m-(i+1)}{2}} n^{\frac13} (\log n)^{\frac23+2\delta+k'-1} \right\}^{2k'-2} \left\{ 2^{\frac{m-(i+1)}{2}} n^{\frac13} (\log n)^{\frac23+2\delta +1} \right\}^{2} \\
&\hspace{4cm} + \cdots + \left\{ 2^{\frac{m-(i+1)}{2}} n^{\frac13} (\log n)^{\frac23+2\delta+k'} \right\}^{2k'} \\
& \le 2^{\frac32} \cdot \left\{ 2^{\frac{m-(i+1)}{2}} n^{\frac13} (\log n)^{\frac23+2\delta+k'} \right\}^{2k'}.
\end{split}\ees
Plugging
this in \eqref{eq:li:bd} completes the induction step, proving our claim \eqref{eq:claim}.
\end{proof}

The proof of the theorem follows
by putting $i=0$ and substituting for $\ell$ and $w$ in \eqref{eq:claim}.
\qed

\section{Proofs of Theorem~\ref{thm:clt} and Theorem~\ref{thm:strip}} \label{sec:CLT}
 Fix $0<\delta < 1/12$ and choose $n$ so that $w = w(\ell) = n^{\frac23}(\log n)^{\frac13+\delta}$.
We divide the rectangle into alternating {\it long} and {\it short} blocks, starting with a long block, where the long blocks
 have length $\ell^a$ and the short
blocks have length $n$. The last block would have the remaining length. Here $a$ is chosen so that
\bes
\frac{2+3\gamma}{4} < a< 1.
\ees
The number of long and short blocks is $m=\lfloor \ell/(\ell^a+n)\rfloor$. Since $w\le \ell^\gamma$ we have
\bes
c\ell^{1-a} \le m \le  C\ell^{1-a}
\ees
for some positive constants $c, C$. Let $X_1, X_2,\cdots, X_m$ denote the maximal lengths of the
(restricted) paths
in the consecutive long boxes.
Let $\vx_i, \vy_i$ be arbitrary left and right endpoints in the $i$th short box.  By an argument similar to \eqref{eq:l:ineq}
one can conclude
\bes
\sum_{i=1}^m  \min_{\vx_i, \vy_i} L^R(\vx_i, \vy_i) \le L^R - \sum_{i=1}^m X_i \le \sum_{i=1}^m \max_{\vx_i,\vy_i} L^R(\vx_i, \vy_i).
\ees

It follows from this that for any fixed $k\ge 1$,
\begin{align}\label{eq:l-x}
\begin{split}
\left\| \left[L^R- \E L^R \right] -\sum_{i=1}^m \left[ X_i - \E X_i\right]\right\|_k
& \le \sum_{i=1}^m \left\|  \max_{\vx_i,\vy_i} L^R(\vx_i, \vy_i) - \E  \max_{\vx_i,\vy_i} L^R(\vx_i, \vy_i) \right\|_k \\
& \qquad + \sum_{i=1}^m \left\| \max_{\vx_i,\vy_i} L^R(\vx_i, \vy_i) -  \min_{\vx_i,\vy_i} L^R(\vx_i, \vy_i) \right\|_k \\
& \le C(\delta,k)\cdot m  n^{\frac13} (\log n)^{\frac23+2\delta}
\end{split}
\end{align}
for $n$ large enough, by an application of Proposition \ref{prop:Delta:bd} and Lemma \ref{lem:xns:bd}.
It also follows by an application of the triangle inequality that
\begin{equation}\label{eq:std:lnr}
\left|  \, \sqrt{\var (L_N^R)} - \sqrt{\sum_{i=1}^m\var ( X_i)}  \,\right| \le C(\delta)\cdot m  n^{\frac13} (\log n)^{\frac23+2\delta}.
\end{equation}
We need the following lemma to complete the proof of Theorem \ref{thm:clt}. The lemma shows that the sum of the $X_i$'s satisfy a Gaussian limit theorem.
\begin{lemma} \label{lem:clt:x} With the notation as above we have as $\ell \to \infty$:
\begin{align}
\frac{\sum_{i=1}^m \left[ X_i - \E X_i \right]}{\sqrt{\sum_{i=1}^m \var(X_i)}} \Rightarrow N(0,1).
\end{align}
\end{lemma}
\begin{proof} We check Lindeberg's condition for proving the central limit theorem. It is sufficient to check that
$\sum_{i=1}^m \E \left[T_i^4 \right] \rightarrow 0$ where
\bes
T_i = \frac{X_i - \E X_i }{\sqrt{\sum_{i=1}^m \var(X_i)}}.
\ees
Using Theorem~\ref{thm:mom:ubd} for the fourth moment and
 Theorem~\ref{thm:var:lbd} for the variance for $X_i$  one has
\bes
\left \|X_i -\E X_i \right\|_4  \le \frac{\big(\ell^a\big)^{\frac12}}{w^{\frac14}}(\log w)^{4}  \quad \text{ and }\quad
 \var(X_i)  \ge \left(\frac{\ell^a}{w^{\frac12}}\right)^{1-\eps}.
\ees
Therefore
\bes
\sum_{i=1}^m \E\left[T_i^4 \right] = \frac{1}{m} \cdot\frac{ \|X_1-\E X_1 \|_4^4}{ \var(X_i)^2} \le   \frac{\ell^{2ac\eps}(\log w)^{16}}{\ell^{1-a} w^{c\eps}}
\ees
which tends to $0$ because $a<1$ and $\eps$ is arbitrary.
\end{proof}
It is now not difficult to prove \eqref{eq:clt}.
Indeed \eqref{eq:l-x} and Theorem \ref{thm:var:lbd} gives
\bes
\left\| \frac{\left[L^R- \E L^R \right] -\sum_{i=1}^m \left[ X_i - \E X_i\right]}{\sqrt{\sum_{i=1}^m \var(X_i)}}\right\|_k
\le \frac{C(\delta,k) m n^{\frac13} (\log n)^{\frac23+2\delta}}{\sqrt{m \cdot \big(\ell^a / w^{\frac12} \big)^{1-\eps}}}
 \longrightarrow 0
\ees
since $a>\frac{(2+3\gamma)}{4}$ and $\eps>0$ is arbitrary. Also by \eqref{eq:std:lnr} and Theorem \ref{thm:var:lbd}
\bes
\sqrt{\frac{\var (L^R)}{\sum_{i=1}^m \var(X_i)} } \to 1.
\ees
This is more than enough to prove Theorem \ref{thm:clt}. \qed

Finally we present the proof of Theorem \ref{thm:strip} stated in the introduction.
\begin{proof}[Proof of Theorem~\ref{thm:strip}] For this proof, we switch back to the usual $(x,y)$ coordinate system.  Let $0<\eps<\frac{3\gamma}{2}$ so that $(1-\frac{\gamma}{2}) \cdot \frac{1-\epsilon}{2}> \epsilon +\frac{\gamma}{2}$. Consider the square $[0,n]^2$ and the three regions
$S_1$, $R$ and $S_2$ in the extended strip as shown below. The region $S_1$ is the rectangular region from the origin
to the first anti-diagonal line at a distance $\sqrt 2 n^{\frac{3\gamma}{2}-\eps}$. The region $S_2$
is the corresponding region on the top right. Thus the length of the middle diagonal rectangle is
$\sqrt 2 n - 2\sqrt 2 n^{\frac{3\gamma}{2}-\eps}$.
\begin{figure}[htbp]
   \centering
   \includegraphics[width=300pt]{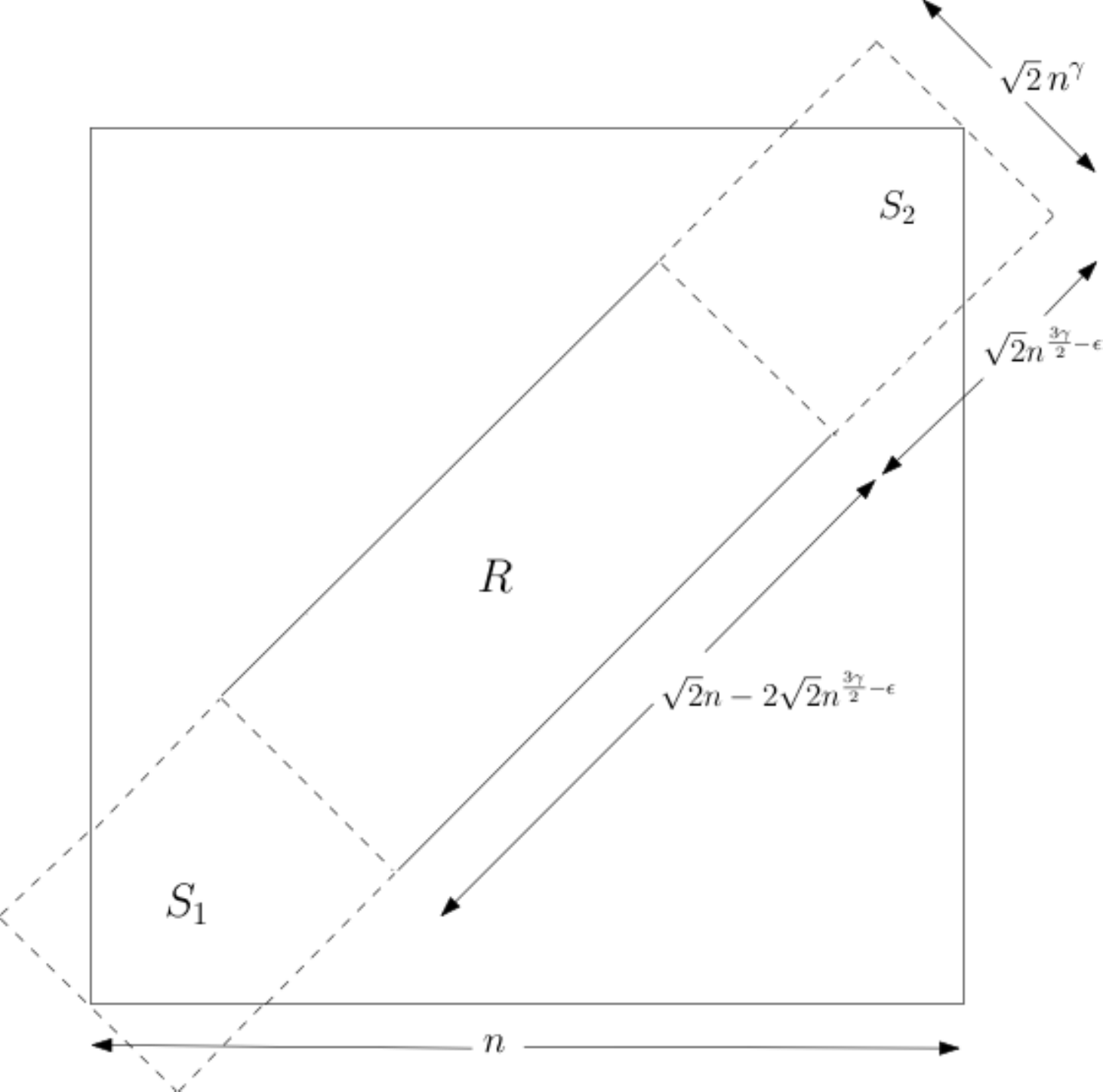}
   \caption{The regions $S_1, R$ and $S_2$.}\label{fig:partition of strip for main theorem}
  \end{figure}
Clearly we have
\be \label{eq:lng}
\min_{\vx_1,\vy_1} L^{S_1} \big(\vx_1, \vy_1\big) + L^R + \min_{\vx_2, \vy_2} L^{S_2} \big(\vx_2, \vy_2\big) \le L_n^{(\gamma)} \le  \max_{\vx_1, \vy_1} L^{S_1} \big( \vx_1, \vy_1\big) + L^R + \max_{\vx_2, \vy_2} L^{S_2} \big(\vx_2,  \vy_2\big).
\ee
The points $\vx_1, \vy_1$ are arbitrary points on the left and right boundaries of $S_1$ and similarly $\vx_2,\vy_2$ are arbitrary boundary points for $S_2$. In particular this gives us by an application of Proposition \ref{prop:Delta:bd}
\bes
 \left \vert \E L_n^{(\gamma)} - \E L^R - \E L^{S_1} \big(\mathbf{0}, (n^{\frac{3\gamma}{2}-\eps},n^{\frac{3\gamma}{2}-\eps}) \big)- \E L^{S_2}\big( (n-n^{\frac{3\gamma}{2}-\eps}, n-n^{\frac{3\gamma}{2}-\eps}), (n,n)\big)\right \vert \le  C n^{\eps +\frac{\gamma}{2}}.
\ees
The first assertion in \eqref{eq:e:var} follows from this, \eqref{eq:e:lr} and \eqref{eq:E:asy} . Using \eqref{eq:lng} and applying Proposition \ref{prop:Delta:bd}
and Lemma \ref{lem:xns:bd} we get
\bes
\left \vert \sqrt{\var L_n^{(\gamma)}} -\sqrt{\var L^R}\right\vert  \le C n^{\eps+\frac{\gamma}{2}}.
\ees
The second assertion in \eqref{eq:e:var} follows from this. This is because of Theorem \ref{thm:mom:ubd} and the bound $\sqrt{\var L^R}\gg n^{\eps+\frac{\gamma}{2}}$, which follows from Theorem \ref{thm:var:lbd} and our choice of $\eps$. For the final statement \eqref{eq:gauss:lim} note that
\bes
\frac{L_n^{(\gamma)} -\E L_n^{(\gamma)}}{\sqrt{\var L_n^{(\gamma)}}} = \sqrt{\frac{\var L^R}{\var L_n^{(\gamma)}}} \cdot  \left[\frac{L^R-\E L^R}{\sqrt{\var L^R}} +\frac{\mathcal{E}}{\sqrt{\var L^R}}\right].
\ees
Here the error term $\mathcal{E}$ is of order $n^{\eps+\frac{\gamma}{2}}$ by \eqref{eq:lng}, Proposition \ref{prop:Delta:bd} and Lemma \ref{lem:xns:bd}, and hence small with respect
to $\sqrt{\var L^R}$. The above argument also shows that $\var L^R /\var L_n^{(\gamma)} \to 1$
as $n\to \infty$. This gives the central limit theorem for $L_n^{(\gamma)}$.
\end{proof}

\section{Open problems} \label{sec:open}
In this section, we collect a few questions which are open.
\begin{enumerate}
\item \label{var:qn} Can one improve the results in Theorem \ref{thm:var:lbd} and Theorem \ref{thm:mom:ubd} and get
a more precise result for  $\var(L^R)$ in rectangles $R$ of length $\ell$ and width $1\ll w\ll \ell^{\frac23}$ ?
In particular is it true that
\bes
\var(L^R) = c \frac{\ell}{w^{\frac12}} \big(1+o(1)\big)
\ees
as $\ell \to \infty$ for an appropriate constant $c>0$?
\item Considering Theorem \ref{thm:exp}, can one get sharper results for $\E L^R$ for a diagonal rectangle $R$ of
length $\ell$ and width $1\ll w \ll \ell^{\frac23}$?  Do there exist positive constants $c_1, c_2$ such that
\bes
\E(L^R) = \sqrt 2 \ell - c_1 \frac{\ell}{w} + c_2 \frac{\ell^{\frac12}}{w^{\frac14}} \big( 1+ o(1) \big)
\ees
as $\ell \to \infty$? Note that $\ell^{\frac12} / w^{\frac14}$ is our prediction in \ref{var:qn} for the standard deviation.
\item Theorem \ref{thm:strip} gives a Gaussian limit for $L_n^{(\gamma)}$ when $\gamma<\frac23$ whereas the results of
Baik, Deift and Johansson \cite{baik-deif-joha-99} imply a Tracy-Widom limit for $L_n^{(\gamma)}$ when $\gamma>\frac23$. What is the limiting distribution
when we take the width of the strip to be $\alpha\cdot n^{\frac23}$ for fixed $\alpha >0$?
\item As pointed out in the introduction, a generic $k$-point configuration in the square $[0,n]^2$ corresponds naturally
to a permutation $\pi$ of $\{1, 2, \cdots, k\}$. Our results then fit within the framework of studying random permutations
with a band structure, i.e. permutations where $| \pi(i)- i  | $ is typically much smaller than $n$. Other models of this type
include the interchange (or stirring) process on a finite segment (introduced in~\cite{toth}), the Mallows model (introduced in~\cite{mall}), the one-dimensional case of displacement-biased random permutations on the lattice~\cite{betz2014random, fyodorov2018band} and the model of $k$-min permutations~\cite{travers2015inversions}.

The longest increasing subsequence was analyzed in two of the above examples: the Mallows model~\cite{bhat-pele, basu2017limit} (drawing on~\cite{muel-star}) and the $k$-min permutation~\cite{travers2015inversions}. The results of these studies, however, are not as detailed as the ones obtained here in the sense that the second-order correction to the expectation, the order of magnitude of the variance and the limit law have not been determined (for band widths growing with the size of the permutation). It is natural to expect that the longest increasing subsequence in many random permutation models with a band structure exhibits similar behavior to the one obtained here and it is of great interest to obtain such results for a general class of models.

\end{enumerate}
\textbf{Acknowledgements:} We thank Lucas Journel for a careful reading of an earlier draft and for suggesting several improvements. We thank Eitan Bachmat for interesting discussions of related problems and application areas. Most of this work was completed while M.J. was at the University of Sheffield, and he thanks the School of Mathematics and Statistics for a supportive environment. The work of R.P. was supported in part by Israel Science Foundation grant 861/15 and the European Research Council starting grant 678520 (LocalOrder).

\bibliographystyle{plain}
\bibliography{lis}

\end{document}